\documentclass[11pt]{amsart}
\usepackage{amsmath}
\usepackage{bbm}
\usepackage{mathrsfs}
\usepackage[english]{babel}
\usepackage[utf8x]{inputenc}
\usepackage[T1]{fontenc}
\usepackage{wrapfig}

\usepackage[top=3cm,bottom=2cm,left=3cm,right=3cm,marginparwidth=1.75cm]{geometry}

\usepackage{amsmath,amssymb, amsthm}
\usepackage{graphicx}
\usepackage[colorinlistoftodos]{todonotes}
\usepackage[colorlinks=true, allcolors=blue]{hyperref}
\usepackage{cite}
\usepackage{enumitem} 
\usepackage[all,cmtip]{xy}
\usepackage{tikz-cd}
\usepackage{lmodern}
\usepackage{pgfplots}
\usetikzlibrary{intersections}
\usepackage{relsize}

\newcommand{\dZ}{{\mathrm{d}_Z^*}}
\newtheorem{theorem}{Theorem}
\newtheorem{proposition}[theorem]{Proposition}
\newtheorem{corollary}[theorem]{Corollary}
\newtheorem{lemma}[theorem]{Lemma}
\newtheorem{definition}[theorem]{Definition}

\newtheoremstyle{claim}% name
  {\topsep}% space above
  {\topsep}% space below
  {}% body font
  {}% indent amount
  {\itshape}% theorem head font
  {}% punctuation after theorem head
  {.5em}% space after theorem head
  {\thmname{#1}\thmnumber{ #2}\thmnote{ (#3)}}% theorem head spec
\theoremstyle{claim}

\theoremstyle{remark} 
\newtheorem{remark}[theorem]{Remark}
\newtheorem{example}[theorem]{Example}

\usepackage{mathtools} 
\mathtoolsset{showonlyrefs,showmanualtags}
\numberwithin{equation}{section}
\usepackage{amsmath}

\newcommand{\R}{\mathbb{R}}

\newcommand{\f}{\varphi}
\newcommand{\di}{\deg_{\mathrm{iso}}}

\renewcommand{\a}{\alpha}
\renewcommand{\t}{\theta}
\newcommand{\e}{\varepsilon}
\newcommand{\coo}[3]{\mathcal{C}^{#1}(#2,#3)}
\newcommand{\de}{\partial}
\newcommand{\nurm}[3]{\|#2\|_{C^{#1}(D,\R^{#3})}}
\newcommand{\be}{\begin{equation}}
\newcommand{\ee}{\end{equation}}

\title{What is the degree of a smooth hypersurface?}

\author{Antonio Lerario}
\address{SISSA, Via Bonomea 265, 34136 Trieste, Italy}
\email{lerario@sissa.it}

\author{Michele Stecconi}
\address{SISSA, Via Bonomea 265, 34136 Trieste, Italy}
\email{stecconi.michele@gmail.com}

\begin{document}
\begin{abstract} 
Let $D$ be a disk in $\R^n$ and $f\in C^{r+2}(D, \R^k)$. We deal with the problem of the algebraic approximation of the set $j^{r}f^{-1}(W)$ consisting of the set of points in the disk $D$ where the $r$-th jet extension of $f$ meets a given semialgebraic set $W\subset J^{r}(D, \R^k).$ We call such sets \emph{type}--$W$ \emph{singularities};  examples of sets arising in this way are the zero set of $f$, or the set of its critical points. 

Under some transversality conditions, we prove that $f$ can be approximated with a polynomial map $p:D\to \R^k$ such that the corresponding singularity is diffeomorphic to the original one, and such that the degree of this polynomial map can be controlled by the $C^{r+2}$ data of $f$. More precisely, denoting by $\Delta_W\subset C^{r+1}(D, \R^k)$ the set of maps whose $r$-th jet extension \emph{is not} transverse to $W$, we show that there exists a polyomial $p$ such that:
\be\label{eq:ab1}j^{r}p^{-1}(W)\sim j^{r}f^{-1}(W)\quad \textrm{and}\quad \deg(p)\leq O\left(\frac{\|f\|_{C^{r+2}(D, \R^k)}}{\mathrm{dist}_{C^{r+1}}(f, \Delta_W)}\right).\ee
(Here ``$\sim$'' means ``ambient isotopic'' and the implied constant depends on the size of the disk). The estimate on the degree of $p$ implies an estimate on the Betti numbers of the singularity, however, using more refined tools introduced in \cite{MSphd}, we prove a similar estimate, but involving only the $C^{r+1}$ data of $f$.

For a given $W\subset J^{r}(D, \R^k)$, we introduce a notion of $C^\ell$ condition number $\kappa^{(\ell)}_W(f, D)$ of a type--$W$ singularity of a map $f\in C^{\ell}(D, \R^k)$, $\ell\geq r+1$, and we  show that our estimates, as \eqref{eq:ab1}, can be equivalently stated using this notion.
%prove an infinite--dimensional condition number theorem:
%\be \kappa^{(\ell)}_W(f, D)=\Theta\left(\frac{\|f\|_{C^{\ell}(D, \R^k)}}{\mathrm{dist}_{C^{r+1}}(f, \Delta_W)}\right).\ee

These results specialize to the case of zero sets of $f\in C^{2}(D, \R)$, and give a way to approximate a smooth hypersurface defined by the equation $f=0$ with an algebraic one, with controlled degree (from which the title of the paper). In this case we deal both with the approximation of $f=0$ inside the disk $D$ and with its global approximation in $\R^n$. As a corollary we prove an upper bound on the number of diffeomorphism classes of compact hypersurfaces with a bounded condition number. 

Moreover, we also deal with the more basic problem of producing an actual equation $f=0$ for a given compact hypersurface $Z\subset \R^n$, and we control the condition number of this equation with the geometric data of $Z$ (its reach and its diameter). 
We prove that the Betti numbers of the hypersurface can be estimated with the $\kappa^{(1)}$ condition number of the defining equation and we show that the order of this estimate is sharp. 

In particular, combining these results, we show that a compact hypersurface $Z\subset D\subset \R^n$ with positive reach $\rho(Z)>0$  is isotopic to the zero set in $D$ of a polynomial $p$ of degree
\be \deg(p)\leq c(D)\cdot 2 \left(1+\frac{1}{\rho(Z)}+\frac{5n}{\rho(Z)^2}\right),\ee
where $c(D)>0$ is a constant depending on the size of the disk $D$ (and in particular on the diameter of $Z$).

\end{abstract}
\maketitle
\tableofcontents
\section{Introduction}
\subsection{Polynomial approximation of singularities of smooth maps}In this paper we deal with the following problem: given a smooth (i.e. sufficiently regular) function $f:D\to \R$ defined on a disk $D\subseteq \R^n$ and whose zero set $Z(f)$ is a smooth compact manifold, what is the smallest degree of a polynomial $p$ whose zero set $Z(p)$ is diffeomorphic to $Z(f)$? 

More generally, we will consider the problem of the polynomial approximation of nondegenerate singularities of smooth maps: given a closed and stratified subset $W$ of the jet space\footnote{Here $J^{r}(D, \R^k)$ denotes the $r$-th jet bundle of maps $f:D\to \R^k$ and, given $f:D\to\R^k$ of class $C^r$, $j^rf:D\to \R^k$ denotes its $r$-th jet extension (see \ref{sec:jets} for more details). } $J^{r}(D, \R^k)$ and a smooth map $f:D\to \R^k$ transverse to all the strata of $W$, what is the smallest degree of a polynomial map $p:D\to \R^k$ such that the two pairs $(D, j^rf^{-1}(W))$ and $(D, j^rp^{-1}(W))$ are diffeomorphic?

Besides the case of hypersurfaces, which corresponds to the choice of $W=D\times\{0\}\subset J^0(D, \R)=D\times \R$, other examples of special interests covered by this framework are: systems of smooth inequalities, corresponding to the case of $W=D\times C\subset J^0(D, \R^k)=D\times \R^k$, where $C$ is a closed polyhedral cone; critical points of a function $f:D\to \R$, corresponding to the choice $W=D\times \R\times \{0\}\subset J^1(D, \R)=D\times \R\times \R^n;$ critical points of a smooth map $f:D\to \R^n$, with the choice $W=D\times \R^k\times \{\det=0\}\subset J^1(D, \R^n)=D\times \R^n\times \R^{n\times n}.$ In general we will call the set $j^rf^{-1}(W)$ a \emph{type--$W$ singularity}.
%See Appendix \ref{sec:examples} for more details and examples.
%\comm{Do we want to put this appendix?}

In order to answer the above questions, we shall adopt first a geometric approach. We make the assumption, verified in all cases of practical interest, that $W\subseteq J^{r}(D, \R^k)$ is a semialgebraic set. Note that this does not mean that the singularity is semialgebraic, but rather that it is given by semialgebraic conditions on the derivatives of a \emph{smooth} map. Given such $W$, we denote by $\Delta_W\subset C^{r+1}(D, \R^k)$ the set:
\be \Delta_{W}=\left\{\textrm{$f\in C^{r+1}(D, \R^k)$ such that $j^rf:D\to J^{r}(D, \R^k)$ \emph{is not} transverse to $W$}\right\}.\ee 
Here transversality means with respect to all the strata of a given fixed Whitney stratification of $W$, both for $j^rf$ and $(j^rf)|_{\partial D}$. The set $\Delta_W$ acts as a discriminant for our problem,  and the jet of a map $f\in C^{r+1}(D, \R^k)$ is transverse to $W$ if and only if $\mathrm{dist}_{C^{r+1}}(f, \Delta_W)>0.$ When both $j^rf$ and $(j^rf)|_{\partial D}$ are transverse to $W$, we will simply write $j^rf\pitchfork W$. In this case the set $j^rf^{-1}(W)\subseteq D$ is a Whitney stratified subcomplex of the disk and we will refer to it as a \emph{nondegenerate singularity}; for instance, if $W$ is a smooth submanifold, then so is $j^rf^{-1}(W).$ 
%When the transversality assumption is verified, we call $j^{r}f^{-1}(W)$ a \emph{nondegenerate singularity}.

Given subcomplexes $K_0$ and $K_1$ of the disk, we will say that the two pairs $(D, K_0)$ and $(D, K_1)$ are \emph{isotopic}, and write $(D, K_0)\sim (D, K_1)$, if there exists a continuous family of diffeomorphisms $\varphi_t:D\to D,$ with $t\in [0,1]$, such that $\varphi_0=\mathrm{id}_D$ and $\varphi_1(K_0)=K_1.$ With this notation, our first result is the following.
\begin{theorem}\label{thm:mainapp} Let $W\subseteq J^{r}(D, \R^k)$ closed and semialgebraic. For every $f\in C^{r+2}(D, \R^k)$ with $j^{r}f\pitchfork W$ there exists a polynomial map $p=(p_1, \ldots, p_k)$ with each $p_i\in \R[x_1, \ldots, x_n]$ with 
\be\label{eq:boundapp} \deg(p_i)\leq c_1(r,D)\max \left\{r+1, \frac{\|f\|_{C^{r+2}(D, \R^k)}}{\mathrm{dist}_{C^{r+1}}(f, \Delta_W)}\right\}\ee
and such that:
\be\label{eq:isotopy} (D, j^{r}f^{-1}(W))\sim (D, j^{r}p^{-1}(W)).\ee
(Here $c_1(r,D)$ is a constant depending only on the size of the disk $D$ and on $r$).
%$c_1(r,D)=\max\{1,2a_{r+2}(D)\}$
\end{theorem}

The transversality assumption in the previous statement is necessary to prevent pathological situations. For instance, every closed set in $D$ is the zero set of a smooth function,  thus there exists a smooth $f$ such that $f^{-1}(0)$ equals the Cantor set; however in this case the pair $(D, f^{-1}(0))$ cannot be diffeomorphic to a pair $(D, p^{-1}(0))$ with $p$ a polynomial. 

From the previous result, using standard techniques from real algebraic geometry, one can immediately produce an upper bound on the topological complexity of a nondegenerate singularity, measured by the sum of its Betti numbers.
\begin{corollary}\label{coro:betti}Given $W\subseteq J^{r}(D, \R^k)$ closed and semialgebraic there exists $c_2(W)>0$ such that for every $f\in C^{r+2}(D, \R^k)$ with $j^{r}f\pitchfork W$ we have:
\be\label{eq:bettiboundcoro} b(j^{r}f^{-1}(W))\leq c_2(W)\cdot \max\left\{r+1, \frac{\|f\|_{C^{r+2}(D, \R^k)}}{\mathrm{dist}_{C^{r+1}}(f, \Delta_W)}\right\}^n.\ee
\end{corollary}
\subsection{$C^{r+1}$ bound for Betti numbers}
Note that in \eqref{eq:bettiboundcoro}, the $C^{r+2}$ norm of the function $f$ appears, even if the transversality condition would require only looking at the $(r+1)$-th jet of $f$. This comes from the fact that the error estimate for the polynomial approximation of $f$ in the $C^{r+1}$ topology involves its $C^{r+2}$ norm (this condition can be slightly relaxed when working with Lipschitz derivatives). However, if one is only interested in bounding the topology of $j^rf^{-1}(W)$, it turns out that an estimate on the $C^{r+1}$ norm of $f$ suffices, at least in the case $W$ is smooth, as we will prove in Theorem \ref{thm:introsingsemiwitdash} below. In the case of hypersurfaces, we will discuss this in more details in Section \ref{sec:semicont}.
\begin{theorem}\label{thm:introsingsemiwitdash}
Let $W\subseteq J^{r}(D, \R^k)$ be a smooth, compact and semialgebraic submanifold with the property that $W\pitchfork J^r_z(D,\R^k)$ for all $z\in D$. 
There exists a constant $c_3(W)>0$ such that for every $f\in C^{r+1}(D, \R^k)$ with $j^{r}f\pitchfork W$ and $j^rf^{-1}(W)\cap \de D =\emptyset$ we have:
% \be\label{eq:introsemibettising}
% b(j^{r}f^{-1}(W))\leq c_3(W)\cdot \left(1+a_{r+1}(D)\frac{\nurm {r+1}fk}{\hat{\delta}_W(f,D)}\right)^n.\ee
\be\label{eq:introsemibettising}
b(j^{r}f^{-1}(W))\leq c_3(W)\cdot\max\left\{r,  \frac{\nurm {r+1}fk}{\hat{\delta}_W(f,D)}\right\}^n.
\ee
\end{theorem}
The quantity $\hat{\delta}_W(f,D)$ is defined in Section \ref{sec:singbetti}. In analogy with the $C^{r+1}$ distance form the discrimant appearing in the estimate \eqref{eq:bettiboundcoro}, it  depends only on $W$ and on the $(r+1)$-{th} jet of the map $f$. In particular we have $\hat{\delta}_W(f,D)>0$ if and only if $j^{r}f\pitchfork W$ and $j^rf^{-1}(W)\cap \de D =\emptyset$. Notice the different type of boundary condition: the transversality assumption ``$(j^rf)|_{\de D}\pitchfork W$'' in Theorem \ref{coro:betti} is replaced here with the stronger requirement that there is no point $z\in \de D$ such that $j^rf(z)\in W$.

What is interesting about Theorem $\ref{thm:introsingsemiwitdash}$ is that the right hand side of \eqref{eq:introsemibettising} does not depend on the size (nor even on the existence) of the derivatives of $f$ of order higher than $r+1$. This is obtained by means of Theorem \ref{thm:semiconti}, a recent result from \cite{mttrp} that says that given a function $f$ such that $j^rf\pitchfork W$, the Betti numbers of a singularity $(j^rg)^{-1}(W)$ are bigger than those of $(j^rf)^{-1}(W)$, provided that $g$ is close enough to $f$ in the $C^r$ topology. 
In this way we can control directly the Betti numbers of $j^rg^{-1}(W)$ and relax the requirements on the needed polynomial approximation, since we can bypass the isotopy condition \eqref{eq:isotopy} of Theorem \ref{thm:mainapp}.
\subsection{The condition number of a singularity}It is often more convenient to substitute the distance from the discriminant with a more explicit quantity $\delta_W(f, D)$, defined as follows. Denote by $\Sigma_{W, z}\subset J_z^{r+1}(D, \R^k)$ the set of all possible jets of maps which are not transverse to $W$ at $z$. This is a closed and semialgebraic subset of $J^{r+1}_z(D, \R^k)$ and we define the number:
\be \delta_W(f, D)=\min_{z\in D}\mathrm{dist}(j^{r+1}f(z), \Sigma_{W, z}).\ee
Observe that when $z\in \partial D$ the set $\Sigma_{W, z}$ consists of ``two pieces'': in fact if $z\in \mathrm{int}(D)$  we only have to consider the transversality of $j^rf$, while if $z\in \partial D$ one needs also to take into account the transversality of $(j^rf)|_{\partial D}$, which involves a restriction of the jet, as a multilinear map, to $T_z(\partial D)\simeq \{z\}^\perp.$ 

The transversality of $j^rf$ in most cases has a simple geometric interpretation, and for $z\in \mathrm{int}(D)$ the set $\Sigma_{W, z}$ can be easily described. For example: in the case of hypersurfaces, $\Sigma_{W, z}=\{z\}\times \{0\}\times \{0\}\subset D\times \R\times \R^n$ and $\mathrm{dist}(j^{1}f(z), \Sigma_{W, z})=\left(|f(z)|^2+\|\nabla f(z)\|^2\right)^{1/2};$ in the case of critical points of a smooth function $f:D\to \R$, we have that $\Sigma_{W, z}=\{z\}\times \R\times \{0\}\times \{\det=0\}\subset D\times \R\times \R^n\times \mathrm{Sym}(n, \R)$ and 
$\mathrm{dist}(j^{2}f(z), \Sigma_{W, z})=\left(\|\nabla f(z)\|^2+\left(\sigma_{1}(\mathrm{He}(f)(z))\right)^2\right)^{1/2},$ where $\mathrm{He}(f)(z)$ denotes the Hessian of $f$ at $z$ and $\sigma_{1}$ denotes the smallest singular value.
%see Appendix \ref{sec:examples}. \comm{Will this still exist?}

The quantity $\delta_W(f, D)$ vanishes exactly when $f\in\Delta_W$, and $j^rf\pitchfork W$ if and only if $\delta_W(f, D)>0$ (Lemma \ref{lemma:transv} below). The key property, that allows to translate estimates in terms of this quantity into the above geometric framework, is the fact that the distance to the discrminant and the quantity $\delta_W(f, D)$ have the same order of magnitude. In fact, in Proposition \ref{propo:delta}, we show that there exists a constant $C=C(n,k,r)>0$ such that: 
\be\label{eq:constant} \delta_W(f, D)\leq \mathrm{dist}_{C^{r+1}}(f, \Sigma_W)\leq C\cdot \delta_W(f, D).\ee
The first inequality follows from a characterization of the set of functions not transverse to $W$ at $z$ as those with jet belonging to $\Sigma_{W, z}$; the second inequality uses a classical result of Whitney on the norm of a map with prescribed jet at one point.
\begin{remark}The constant $C>0$ in \eqref{eq:constant} depends also on the choice of the pointwise norm on the jet bundles. In general we have $C>1$, as shown in Example \ref{ex:CgeOne}.
\end{remark}

Using this notation, we introduce the notion of condition number for a map with respect to a closed, stratified $W\subseteq J^{r}(D, \R^k)$ (Definition \ref{def:condition}). Given $f\in C^{\ell}(D, \R^k)$ with $\ell\geq r+1$, we set:
\be\label{eq:condition} \kappa_W^{(\ell)}(f, D)=\frac{\|f\|_{C^{\ell}(D, \R^k)}}{\delta_W(f, D)}\geq \frac{\|f\|_{C^{\ell}(D, \R^k)}}{\mathrm{dist}_{C^{r+1}}(f, \Delta_W)}.\ee
Notice that the numerator in \eqref{eq:condition} depends on the topology we are considering on the space of functions, while the denominator does not. Transversality is synonimous of bounded (i.e. non-infinite) condition number. Because of the inequality on the right hand side of \eqref{eq:condition}, the estimates \eqref{eq:boundapp} and \eqref{eq:bettiboundcoro} can also be stated using $\kappa_W^{(r+2)}(f, D).$
\subsection{The condition number of a defining equation for  a hypersurface}Given a smooth and compact hypersurface $Z$ contained in $\mathrm{int}(D)$, the existence of a polynomial $p$ such that $(D, Z)\sim (D,Z(p))$ is guaranteed by a classical result of Seifert \cite{Seifert}. This problem is also known as the ``algebraic approximation problem'', and has several generalizations, culminating with the celebrated Nash-Tognoli Theorem \cite{Nash, Tognoli}: every smooth and compact manifold $M$ is diffeomorphic to an algebraic set in $\R^n$. (Our problem concerns with the special case $M$ is already a hypersurface in $\R^n$). 
%In fact Nash-Tognoli Theorem consists of two separate results: Nash first proved that $M$ is diffeomorphic to a component of a real algebraic set, and then Tognoli proved that we can choose this algebraic set to be connected. 

The proof of Seifert's Theorem (clearly explained in \cite{Kollar}) consists in first realizing $Z$ as the regular zero set of a function $f:\R^n\to \R$ and then approximating $f$ on the disk $D$ in the $C^1$ topology with a polynomial. In the very first step, one uses the fact that every smooth and compact hypersurface in $\R^n$ is the zero set of a smooth function. This result is well-known, see for instance \cite[Theorem 7.2.3]{DFN}, and a natural question in our framework is to produce an estimate on the condition number of a defining equation for $Z$, in terms of some metric data of the embedding $Z\hookrightarrow \R^n$. 

To this end, given $Z\subset \R^n$ of class $C^1$ we define the \emph{reach} of $Z$ as:
\be \rho(Z)=\sup_{r>0}\left\{\mathrm{dist}(x, Z)<r\implies \exists! z\in Z \,|\,\mathrm{dist}(x, Z)=\mathrm{dist}(x, z)\right\}.
\ee
The reach of a $C^1$ manifold doesn't need to be positive, as shown in \cite[Example 4]{KP}, where for every $0<\epsilon<2$ an example of a $C^{2-\epsilon}$ compact curve with zero reach in $\R^2$ is constructed\footnote{Near the origin this is in fact just the set $Z\subset \R^2$ defined by $Z=\{y=x^{2-\epsilon}\}$}; however $\rho(Z)>0$ if $Z$ is of class $C^2.$ We prove the following result\footnote{In the special case of hypersurfaces, i.e. when $W=D\times \{0\}$, in the notation of the various involved quantities we omit the dependence from $W$ in the subscripts.}.
\begin{theorem}\label{thm:condeq}Given a compact hypersurface $Z\subset \R^n$ of class $C^1$ with $\rho(Z)>0$, there exists a $C^1$ function $f:\R^n\to \R$ whose zero set is $Z$ and such that for every disk $D$ containing $Z$ such that $\mathrm{dist}(Z, \partial D)>\rho(Z)$, we have:
\be\label{eq:k1} \kappa^{(1)}(f, D)\leq 2\left(1+\frac{1}{\rho(Z)}\right).\ee
If moreover $Z$ is of class $C^2$, then the function $f$ can be chosen of class $C^2$ and satisfying:
\be\label{eq:k2} \kappa^{(2)}(f, D)\leq 2\left(1+\frac{1}{\rho(Z)}+\frac{5n}{\rho(Z)^2}\right).\ee
\end{theorem}
\subsection{What is the degree of a smooth hypersurface?} Once a compact $C^2$ hypersurface $Z\subset \R^n$ is given, as we have seen with Theorem \ref{thm:condeq}, we can produce a defining equation with $C^2$ condition number controlled with a function of the reach $\rho(Z)>0$ (recall that the $C^2$ regularity assumption implies that the reach is nonzero). Therefore, assuming $Z=Z(f)$ with $f\in C^2$, we come back to our original question: what is the smallest degree of a polynomial $p$ whose zero set $Z(p)$ is diffeomorphic to $Z(f)$? Here we should make a choice: whether we want to approximate $Z$ only inside the disk $D$, or we want to get a global approximation in $\R^n$. In fact Nash-Tognoli Theorem also consists of two separate results: for a compact $M$, Nash first proved that $M$ is diffeomorphic to a component of a real algebraic set, and then Tognoli proved that we can choose this algebraic set to be connected. 

 We deal with the problem of the global approximation in Appendix \ref{sec:global}, where we prove an explicit (but rather unpleasant) bound on the degree of the global approximating polynomial in terms of the $C^2$ data of the defining function. Concerning the approximation of $Z(f)$ inside the disk $D$, specializing Theorem \ref{thm:mainapp}, we see that there exists a polynomial $p$ with:
\be\label{eq:apphyp} \deg(p)\leq c_1(0,D)\cdot\max\left\{1, \kappa^{(2)}(f, D)\right\}\ee% prima era 2a_2(D)
such that $(D, Z(f))\sim (D, Z(p)).$  Combining this with Theorem \ref{thm:condeq} we get the following result, which uses only the available geometry of $Z$.
\begin{corollary}
Given a compact hypersurface $Z\subset \R^n$ of class $C^2$ and a disk $D$ containing it such that $\mathrm{dist}(Z, \partial D)>\rho(Z)$, there exists a polynomial $p$ with
\be \deg(p)\leq c_1(0,D)\cdot 2 \left(1+\frac{1}{\rho(Z)}+\frac{5n}{\rho(Z)^2}\right)%4a_2(D)
\ee
such that 
\be (D, Z)\sim (D, Z(p)).\ee
\end{corollary}

Another interesting consequence of \eqref{eq:apphyp} is Theorem \ref{thm:dash} below,  which provides a bound on the number $\#(\kappa,D)$ of isotopy classes (and in particular on the number of diffeomorphism classes)  of compact hypersurfaces $Z\subset \mathrm{int}(D)$ defined by a regular equation $Z=Z(f)$ with bounded $C^2$ condition number.
\begin{theorem}\label{thm:dash}
There exist two constants $C_1,C_2$ (depending on $D$) such that the number $\#(\kappa,D)$ of rigid isotopy classes of pairs $(D, Z(f))$ with $Z(f)\subset \mathrm{int}(D)\subset \R^n$ and with $\kappa^{(2)}(f,D)\leq \kappa$ is bounded by:
\be 
\#(k,D)
\leq \min\left\{C_1,\kappa^{\left(C_2 \kappa^{n+1}\right)}\right\}
\ee
\end{theorem}
In the case of the approximation of a hypersurface $Z$ inside the disk $D$, we introduce in Section \ref{sec:isodeg} the notion of \emph{isotopy degree} of $Z$ in $D$, denoted by $\di(Z,D)$: this is the smallest degree of a polynomial $p$ such that $(D, Z(f))\sim (D, Z(p))$. The idea here is that, as soon as you can move from the smooth category to the semialgebraic one, a complete new set of tools becomes available (for instance the use of Thom-Milnor for controlling the topology of the zero set) and Theorem \ref{thm:mainapp} allows to make this transition quantitative.  

An interesting observation is that if $f$ is itself a polynomial of degree $d$, then $\di(Z(f), D)$ might be smaller than $d$. In fact, for most polynomials $f$ of degree $d$ one has $\di(Z(f), D)\leq O(\sqrt{d \log d})$, see Remark \ref{remark:measure}.

 \subsection{Semicontinuity of Betti numbers}\label{sec:semicont}
 If we are interested in providing an upper bound for the sum of the Betti numbers of a compact  $C^2$ hypersurface, we can in principle use Corollary \ref{coro:betti} and get:
 \be b(Z(f))\leq c_2\max\left\{1, \kappa^{(2)}(f, D)\right\}^n,\ee
 for some constant $c_2=c_2(W)>0.$
 However, with some extra work, it is possible to provide a bound on the Betti numbers of $Z(f)$ only using $\kappa^{(1)}(f, D)$, i.e. only using $C^1$-information on $f$, as next Theorem shows.
 \begin{theorem}\label{thm:semiwitdash}
Let $f\in C^1(D, \R)$ such that the equation $f=0$ is regular and all of its solutions belong to the  interior of $D$.
\begin{enumerate}
\item\label{itm:semiwitdash1} There exists a constant $c_4>0$ such that the total Betti number of $Z(f)$ is bounded by:
\be\label{eq:bound3} b(Z(f))\leq c_4\cdot(  \kappa^{(1)}(f, D))^{n}.
\ee
\item\label{itm:semiwitdash2} There exists a \emph{bounded} sequence 
$\{f_m\}_{m\in \mathbb{N}}\subset C^1(D, \R)$ with
\be \lim_{m\to \infty}k^{(1)}(f_m, D)=+\infty\ee
and a constant $c_5>0$ such that for every $m\in \mathbb{N}$ the zero set $Z(f_m)\subset \mathrm{int}(D)$ is regular and
\be b(Z(f_m))\geq c_5\cdot(\kappa^{(1)}(f_m, D))^{n}.
\ee
\end{enumerate}
\end{theorem}
Let us comment this result. First, we observe that it is possible to deduce a result similar to \eqref{eq:bound3} also from the work of Yomdin \cite{Yomdin}, where bounds on the Betti numbers of $Z(f)$ are stated in terms of the distance from zero to the the set of critical values of $f$ (Remark \ref{remark:distance} below compares the two quantities). See also \cite{Yomdin84} for an other result of Yomdin, concerning the Hausdorff measure of $Z(f)$.  

Here \eqref{eq:bound3} is a consequence of a semicontinuity result for the Betti numbers of the zero set of $f$ under small perturbations in the $C^0$ topology (rather than in the $C^1$ topology, which is what one would expect to need). This is discussed in Section \ref{sec:semicont}.

The second part of the statement shows that the bound in \eqref{eq:bound3} is ``sharp'': as we get close to the discriminat $\Delta$, the complexity of $f$ can actually increase as the reciprocal of the distance from $\Delta$.

Combining Theorem \ref{thm:semiwitdash} with Theorem \ref{thm:condeq} we get the following estimate for the Betti numbers of a hypersurface in terms of its reach, if nonzero. (A similar bound can be obtained using the results from \cite{NSW}, see Remark \ref{rem:NSW} below). 

\begin{corollary}\label{coro:NSW}For any disk $D\subset \R^n$ there exists a constant $c_5>0$ such that, for every compact hypersurface $Z\subset D$ of class $C^1$ with $\rho(Z)>0$, we have: \be b(Z)\leq c_5\left(1+\frac{1}{\rho(Z)}\right)^n.\ee
\end{corollary}
\subsection{Acknowledgements} The authors wish to thank Daouda Niang Diatta for many interesting discussions, and Pierre Lairez and Yosef Yomdin for their comments on the first draft of this paper.
\section{Preliminaries}
\subsection{Jet bundles}\label{sec:jets}
We assume that the reader is already familiar with notion of jet space, referring to the textbooks \cite{Hirsch, Arnold} for more details. 
We simply recall that the jet space $J^{r}(D, \R^k)$ can be naturally identified with $D\times  J^r(n,k)$, where 
\be J^{r}(n,k)=\left(\bigoplus_{i=1}^r \R^{n+i\choose i}\right)^k\ee
and $\R^{n+i\choose i}$ is the space of homogeneous polynomials of degree $i$ in $n$ variables (see \cite{Hirsch}). (The identification of a jet with an element of $J^{r}(n,k)$ is made through the list of the partial derivatives). We endow $J^{r}(n, k)$ with the standard euclidean structure, allowing to compute distances between elements in $J^r_z(D, \R^k).$

Using this notation, we can define the $C^r$ norm of $f\in C^r(D, \R^k)$ as
\be\label{eq:norm2} \|f\|_{C^r(D, \R^k)}=\max_{z\in D}\|j^rf(z)\|.\ee
More explicitly, given $f=(f_1, \ldots, f_k)\in C^r(D, \R^k)$  its $C^r$ norm is:
\be\label{eq:norm1}\|f\|_{C^r(D, \R^k)}=\sup_{z\in D}\left(\sum_{i=1}^k\sum_{|\alpha|\leq r}\left|\frac{\partial ^\alpha f_i}{\partial x^{\alpha}}(z)\right|^2\right)^{1/2}.\ee

We observe that the definition of $C^r$ norm is sensitive to the choice of the norm in the fibers of the jet bundle. For instance, if one replaces in \eqref{eq:norm2} the quantity $\|j^rf(z)\|$ with the sup of the value of each partial derivative of $f$ order at most $r$ at $z$, one gets an equivalent $C^r$ norm. The choice of another fiberwise norm results in different constants appearing in the theorems below, which would have similar statements.

\subsection{Stratifications} Recall from  \cite[Section 9.7]{BCR} that given two disjoint and connected Nash submanifolds (i.e. smooth and semialgebraic) $X$ and $Y$ in $\R^n$, such that $Y\subset \mathrm{clos}(X)$, we say that they satisfy the Whintey condition (a) if for every point $y\in Y$ and for every sequence $\{x_n\}_{n\geq 0}\subset X$ such that $\lim x_n=y$ and $\lim_n T_{x_n}X=\tau\in G(\dim(X), n)$, then $\tau$ contains $T_yY.$

Given a semialgebraic set $W\subseteq \R^p$ we will say that this set it is \emph{Whitney stratified} if it is stratified as $W=\coprod W_j$ with each stratum a Nash submanifold such that, whenever $W_{j_1}\subset \mathrm{clos}(W_{j_2})$ for some pair of strata $W_{j_1}$ and $W_{j_2}$, then they satisfy Whintey condition (a). Recall that, by \cite[Theorem 9.7.11]{BCR}, every semialgebraic stratification can be refined to a Whitney stratification.

\begin{definition}Let $W\subseteq J^{r}(D, \R^k)$ be closed, semialgebraic and Whitney stratified:
\be\label{eq:Whitney} W=\coprod_{j=1}^s W_j.\ee
Given $f\in C^{r+1}(D, \R^k)$ we will say that $j^rf$ is transverse to $W$, and write $j^rf\pitchfork W$, if both $j^rf$ and $j^rf|_{\partial D}$ are transverse to all the strata of the stratification \eqref{eq:Whitney}.  
\end{definition}
\begin{proposition}\label{propo:sigma}Given $W\subseteq J^{r}(D, \R^k)$ closed, semialgebraic and Whitney stratified, there exists $\Sigma_W\subset J^{r+1}(D, \R^k)$, closed and semialgebraic, such that for every $f\in C^{r+1}(D, \R^k)$ we have that $j^{r}f$ \emph{is not} transverse to $W$ at $z$ if and only if $j^{r+1}f(z)\in \Sigma_{W}.$
\end{proposition}
\begin{proof}Let $W=\coprod_{j}W_j$ be the given Whitney stratification of $W$. The condition that $j^rf:D\to J^{r}(D, \R^k)$ is not transverse to $W$ at $x$ means that $j^rf(x)\in W_j$ for some $j=1, \ldots, s$ and that:
\be\label{eq:strat}\mathrm{im}\left(d_x(j^rf)\right)+T_{j^rf(x)}W_j\neq T_{j^rf(x)}J^{r}(D, \R^k)\simeq \R^{p}.
\ee
For every stratum $W_j$ let $d_j$ denote its dimension and let $\tau_j:W_j\to \R^{d_j\times p}$ be a semialgebraic map such that for every $w\in W_j$ the columns of $\tau_j(w)$ are an orthonormal basis for $T_{w}W_j $ -- we can assume that $\tau_j$ is continuous, possibly after refining the stratification. Then the condition \eqref{eq:strat} can be written as:
\be \mathrm{rk}(d_xj^rf, \tau_j(j^rf(x)))\leq p-1.\ee
In particular, since both $d_xj^rf$ and $j^rf(x)$ are linear images of $j^{r+1}f(x)$, i.e. $d_xj^rf=\pi_1(j^{r+1}f(x))$ and $j^rf(x)=\pi_2(j^{r+1}f(x))$ for some semialgebraic maps $\pi_1:J^{r+1}(D, \R^k)\to J^1(D, J^r(d, \R^k))$ and  
$\pi_2:J^{r+1}(D, \R^k)\to J^{r}(D, \R^k),$ it follows that the condition that $j^rf$ is not transverse to $W$ can be written as:
there exists $x\in D$ such that $\pi_2(j^{r+1}f(x))\in W_j$ for some $j=1, \ldots, s$ and $\mathrm{rk}(\pi_1(j^{r+1}f(x)), \tau_j(\pi_2(j^{r+1}f(x)))\leq p-1$.

In other words, defining the semialgebraic set:
\begin{align} \Sigma_W^1&=\bigcup_{j=1}^s\pi_2^{-1}(W_j)\cap \left\{\textrm{rk}\left(\pi_1(\cdot), \tau_j(\pi_2(\cdot))\right)\leq p-1\right\}\\
&\label{eq:closed}=\pi_2^{-1}(W)\cap\left(\bigcup_{j=1}^s\left\{\textrm{rk}\left(\pi_1(\cdot), \tau_j(\pi_2(\cdot))\right)\leq p-1\right\}\right),
\end{align}
we see that $j^rf$ is not transverse to $W$ at $x$ if and only if $j^{r+1}f(x)\in \Sigma_W^1$. Analogously we can define a semialgerbaic set $\Sigma_W^2$ such that $j^rf|_{\partial D}$ is not transverse to $W$ if and only if $j^{r+1}f(x)\in \Sigma_W^2$ for some $x\in \partial D.$ The set $\Sigma_{W}$ is defined as the union of $\Sigma_W^1$ and $\Sigma_W^2.$

Let us now prove that $\Sigma_W$ is closed. We will show that $\Sigma_W^1$ is closed, the proof for $\Sigma_{W}^2$ is similar. Since $W$ is closed, it is enough to show that the set in the parenthesis in \eqref{eq:closed} is closed, and this follows immediately form the fact that the stratification for $W$ satisfies Whitney condition (a).
\end{proof}
\subsection{Distance to the discriminant}
\begin{definition}Let $W\subseteq J^{r}(D, \R^k)$ closed, semialgebraic and Whitney stratified. 
We define the quantity:
\be\label{eq:inf} \delta_W(f, D)=\inf_{z\in D}\mathrm{dist}(j^{r+1}f(z), \Sigma_{W, z}),\ee
where we set $\Sigma_{W, z}=\Sigma_W\cap J_z^{r+1}(D, \R^k)$ and we adopt the convention that if $\Sigma_{W, z}=\emptyset$ then $\mathrm{dist}(j^{r+1}f(z), \Sigma_{W, z})=\infty$.
\end{definition}
We introduce now the set of all maps which are not transverse to $W$.
\begin{definition}Let $W\subseteq J^{r}(D, \R^k)$ closed, semialgebraic and Whitney stratified. We denote by $\Delta_W\subset C^{r+1}(D, \R^k)$ the set
\be \Delta_W=\{g\in C^{r+1}(D, \R^k)\,|\, \textrm{$g$ is not transverse to $W$}\},\ee
and, for $z\in D$, by $\Delta_{W, z}(D)$ the set
\be \Delta_{W, z}=\{g\in C^{r+1}(D, \R^k)\,|\, \textrm{$g$ is not transverse to $W$ at $z$}\}.\ee
\end{definition}
Observe that, since transversality is an open condition, the set $\Delta_W$ is closed and
\be\label{eq:union} \Delta_W=\bigcup_{z\in D}\Delta_{W, z}.\ee
Moreover, by Proposition \ref{propo:sigma}, we have:
\be \Delta_{W, z}=\{g\in C^{r+1}\,|\, j^{r+1}g(z)\in \Sigma_{W, z}\}.\ee
Next result relates $\delta_W(f, D)$ with the distance from $\Delta_W$ in the space of $C^{r+1}$ maps. 
\begin{proposition}\label{propo:delta}There exists a constant $C=C(r,k, n)$ such that, for every $W\subseteq J^{r}(D, \R^k)$ closed, semialgebraic and Whitney stratified, and for every $f\in C^{r+1}(D, \R^k)$ we have:
\be \delta_W(f, D)\leq \mathrm{dist}_{C^{r+1}(D, \R^k)}(f, \Delta_W)\leq C \cdot \delta_{W}(f, D).\ee
\end{proposition}
\begin{proof}Let us prove the first part of the statement:
\begin{align}\mathrm{dist}_{C^{r+1}(D, \R^k)}(f, \Delta_W)&=\inf_{s\in \Delta_{W}}\|f-s\|_{C^{r+1}(D, \R^k)}\\
&=\inf_{z\in D}\inf_{s\in \Delta_{W,z}}\|f-s\|_{C^{r+1}(D, \R^k)}& \textrm{using \eqref{eq:union}}\\
&=\inf_{z\in D}\inf_{s\in \Delta_{W,z}}\sup_{x\in D}\|j^{r+1}f(x)-j^{r+1}s(x)\|\\
&\geq\inf_{z\in D}\inf_{s\in \Delta_{W,z}}\|j^{r+1}f(z)-j^{r+1}s(z)\|\\
&\geq\inf_{z\in D}\inf_{s\in \Delta_{W,z}}\mathrm{dist}\left(j^{r+1}f(z), \Sigma_{W, z}\right)&\textrm{by Proposition \ref{propo:sigma}}\\
&=\inf_{z\in D}\mathrm{dist}\left(j^{r+1}f(z), \Sigma_{W, z}\right)\\
&=\delta_W(f, D).
\end{align}
For the second part of the statement we use a classical result by Whitney on the infimum of the $C^{r+1}$ norm of a function with prescribed $(r+1)$-th jet at \emph{one point} (see \cite[Theorem 1]{fefferman} for a modern reference): there exists $C=C(r, k, n)>0$ such that, given $P\in J^{r+1}_z(D, \R^k)$, we have 
\be\label{eq:whitney} \inf \left\{\|g\|_{C^{r+1}(D, \R^k)}\,|\, j^{r+1}g(z)=P\right\}\leq C\cdot \|P\|.\ee
Observe now that:
\begin{align}\inf_{s\in \Delta_{W, z}}\|f-s\|_{C^{r+1}(D, \R^k)}&=\inf\left\{\|f-s\|_{C^{r+1}(D, \R^k)}\,|\, j^{r+1}s(z)\in \Sigma_{W, z}\right\}\\
&=\inf_{\sigma(z)\in \Sigma_{W, z}}\,\inf_{j^{r+1}s(z)=\sigma(z)}\|f-s\|_{C^{r+1}(D, \R^k)}\\
&=\inf_{\sigma(z)\in \Sigma_{W, z}}\left(\inf\left\{\|f-s\|_{C^{r+1}(D, \R^k)}\,|\, j^{r+1}s(z)=\sigma(z)\right\}\right)\\
&=\inf_{\sigma(z)\in \Sigma_{W, z}}\left(\inf\left\{\|g\|_{C^{r+1}(D, \R^k)}\,|\, j^{r+1}g(z)=j^{r+1}f(z)-\sigma(z)\right\}\right)\\
&\leq \inf_{\sigma(z)\in \Sigma_{W, z}}\left(C\cdot \|j^{r+1}f(z)-\sigma(z)\|\right)\\
&\label{eq:last}= C\cdot \mathrm{dist}(j^{r+1}f(z), \Sigma_{W, z}),
\end{align}
where in the first line we have used Proposition \ref{propo:sigma}, and in the fifth line we have used Whitney's result \eqref{eq:whitney}. Taking now the infimum over $z\in D$ on both sides of \eqref{eq:last} concludes the proof.
\end{proof}
\begin{example}\label{ex:CgeOne}
In this example we show that in general $C>1$. Let us consider $D=[-1,1]$ and the ``critical points'' singularity $W\subset J^1(D, \R)$ defined by:
\be W=D\times \R\times \{0\}.\ee
In this way, given $f\in C^1(D, \R)$, we have that $j^1f^{-1}(W)$ consists of the set of critical points for $f$. The pointwise discriminant $\Sigma_{W, z}\subset J^2_z(D, \R)$ is
\be \Sigma_{W, z}=\{z\}\times \R\times \{0\}\times \{0\}.\ee
Let $f(x)=x.$ We claim that $\mathrm{dist}_{C^2}(f, \Delta_W)>\delta_W(f, D).$ To start with, observe that:
\begin{align}\delta_{W}(f, D)&=\inf_{z\in D}\mathrm{dist}(j^2f(z), \Sigma_{W, z})\\
&=\inf_{z\in D}\mathrm{dist}((z,z,1,0), \{z\}\times \R\times \{0\}\times\{0\})\\
&=\inf_{z\in D}\|(z, z, 1,0)-(z,z,0,0)\|=1.
\end{align}
We now assume that in this case $C=1$, which implies $\mathrm{dist}_{C^2}(f, \Delta_W)\leq 1$, and prove that this leads to a contradiction. 

If $C=1$, then for every $\epsilon>0$ there is $g\in\Delta_W$ such that:
\be \|f-g\|_{C^2}<1+\epsilon.\ee
Such a function $g$, being in $\Delta_W$, has a degenerate critical point at $x_0\in D$ and therefore:
\be\|f-g\|_{C_2}\geq \|j^2f(x_0)-j^2g(x_0)\|=\left(\left(x_0-g(x_0)\right)^2+1\right)^{1/2}.\ee
Together with this, the inequality $\mathrm{dist}_{C^2}(f, \Delta_W)\leq 1$ implies now that:
\be |x_0-g(x_0)|^2<\epsilon^2+2\epsilon.\ee
Let us now call $\phi=f-g.$ This function satisfies:
\be\label{eq:ff} |\phi(x_0)|<\sqrt{\epsilon^2+2\epsilon}, \quad \phi'(x_0)=1,\quad \phi''(x_0)=0 \quad \textrm{and}\quad \|\phi\|_{C^2}<1+\epsilon.\ee
We show that the existence for every $\epsilon>0$ of a $C^2$ function satisfying the the list of conditions \eqref{eq:ff} leads to a contradiction. To this end, consider the following differential inequality:
\be y^2+(y')^2\leq a^2,\quad y(0)=1, \quad y'(0)=0.\ee
Then it is easy to see (by direct integration) that:
\be y(x)\geq y(0)\cos x-\sqrt{a^2-y(0)}\sin x.\ee 
Letting now $y=\phi'$ and $a=1+\epsilon$, we get
\be \phi'(x)\geq \cos x-\sqrt{(1+\epsilon)^2-1}\sin x\ee
and, integrating
\be \phi(x)\geq \sin x +\sqrt{\epsilon^2+2\epsilon}(\cos x-1)-\sqrt{\epsilon^2+2\epsilon}.
\ee
This means that
\be(1+\epsilon)^2\geq \|j^2\phi(x)\|^2\geq 1+2(\epsilon^2+2\epsilon)(1-\cos x)-2\sqrt{\epsilon^2+2\epsilon}\sin x+\phi''(x)^2-O(\epsilon^{1/2}).
\ee
In turn this implies:
\be |\phi''(x)|\leq c \epsilon^{1/4},\ee
and consequently
\be \phi'(x)\geq 1-c\epsilon^{1/4}x\quad \textrm{and}\quad \phi(x)\geq x-c\epsilon^{1/4}x^{2}/2.\ee
In particular:
\be \|j^2\phi(1)\|^2\geq 2-O(\epsilon^{1/4}),\ee
which for $\epsilon>0$ small enough is a contradiction.\end{example}
\begin{remark}\label{remark:distance}For a given $f:\R^n\to \R$ of regularity class $C^1$, we can introduce the quantity:
\be \gamma(f,D)=\mathrm{dist}_{\R}(0, f(\mathrm{Crit}(f|_D)),\ee
i.e. the distance from zero to the critical values of $f|_D$.
Denoting as above by $\Delta=\Delta_W$ the set of functions for which $f=0$ is not regular on $D$, for every $f\in C^1(D, \R)$ we have:
\be \delta( f, D)\leq \mathrm{dist}_{C^1(D, \R)}(f, \Delta)\leq \gamma(f, D)\ee
The first inequality is just Proposition \ref{propo:delta}. For the second inequality, let $\gamma(f, D)=\gamma_0>0$ (otherwise $f\in \Delta$ and the inequality is trivial) and pick $\xi\in \R$ with $|\xi|=\gamma_0$ and $\xi$ a critical value of $f|_{D}$. Then $f-\xi\in \Delta$ and $f-t\xi\notin \Delta$ if $|t|<1$. Therefore:
\be \mathrm{dist}_{C^1(D, \R)}(f, \Delta)\leq \|f-(f-\xi)\|_{C^{1}(D, \R)}=\|\xi\|_{C^1(D, \R)}=|\xi|=\gamma_0.\ee
\end{remark}

\begin{definition}\label{def:condition}Let $W\subseteq J^{r}(D, \R^k)$ closed, semialgebraic and Whitney stratified and let $f\in C^{\ell}(D, \R^k)$ with $\ell\geq r+1$.
We define the $C^\ell$-condition number of $f$ on the disk $D$ with respect to $W$ as:
\be \kappa^{(\ell)}_W(f, D)=\frac{\|f\|_{C^\ell(D, \R^k)}}{\delta_W(f, D)}.\ee
\end{definition}
\begin{lemma}\label{lemma:transv}Let $W\subseteq J^{r}(D, \R^k)$ closed and semialgebraic, Whitney stratified. For every $f\in C^{r+1}(D, \R^k)$, we have that $j^rf$ is transverse to $W$ if and only if $\delta_W(f, D)>0$, and in this case $j^rf^{-1}(W)$ is a Whitney stratified set. In particular, if $W$ is smooth, then $j^rf^{-1}(W)$ is a smooth neat submanifold of $D$.
\end{lemma}
\begin{proof}If $j^{r}f$ is transverse to $W$, then for no point $z\in D$ we have $j^{r+1}f(z)\in \Sigma_{W, z}$ by Proposition \ref{propo:sigma}; this means that for every $z\in W$ we have $\mathrm{dist}(j^{r+1}f(z), \Sigma_{W, z})>0$, and since the function $z\mapsto \mathrm{dist}(j^{r+1}f(z), \Sigma_{W, z})$ is continuous on the compact set $D$, then the infimum in \eqref{eq:inf} is a minimum and this minimum must be positive. Viceversa, if $\delta_W(f, D)>0$ then  $\mathrm{dist}(j^{r+1}f(z), \Sigma_{W, z})>0$ for every $z\in D$ and $j^rf$ is transverse to $W$ by Proposition \ref{propo:sigma}.

When $\delta_W(f, D)>0$, the fact that $j^{r}f^{-1}(W)$ is a stratified set is a standard application of the transversality theorems; similarly in the case $W$ is smooth.\end{proof}

\subsection{Quantitative transversality}
\begin{lemma}\label{lemma:qt}Let $W\subseteq J^{r}(D, \R^k)$ closed and semialgebraic, Whitney stratified and $f\in C^{r+1}(D, \R^k)$ such that $j^rf$ is transverse to $W$ (in particular $\delta_W(f, D)>0$). For every $g\in C^{r+1}(D, \R^k)$ such that
\be \|f-g\|_{C^{r+1}(D, \R^k)}<\mathrm{dist}_{C^{r+1}}(f, \Delta_W),\ee the two pairs $(D, j^rf^{-1}(W))$ and $(D, j^rg^{-1}(W))$ are isotopic. In particular, the result holds if $\|f-g\|_{C^{r+1}(D, \R^k)}<\delta_W(f, D)$.
\end{lemma}
\begin{proof}
By assumption on $\|f-g\|_{C^{r+1}}$, the homotopy $f_t=f+t(g-f):D\to \R^k$, for $t\in [0, 1]$ is all disjoint from $\Delta_W$. In particular the corresponding homotopy $j^rf_t:D\to J^{r}(D, \R^k)$ is all transverse to $W$ and the result follows from Thom's Isotopy Lemma \cite[Th\'eor\`eme 2.D.2]{Thom}. As for the second part of the statement: by Proposition \ref{propo:delta} we have $\delta_W(f, D)\leq \mathrm{dist}_{C^{r+1}}(f, \Delta_W)$ and the conclusion follows from the first part.
%We want to show that for every $t\in [0,1]$ the map $j_rf_t$ is transverseto $W$. By Lemma \ref{lemma:transv}, in order to prove this it is enough to prove that $\delta_W(f_t, D)>0$ for all $t\in [0,1]$:
%\begin{align}\delta_W(f_t, D)&=\inf_{z\in D}\mathrm{dist}(j^{r+1}f_t(z), \Sigma_{W, z})\\
%&=\inf_{z\in D}\mathrm{dist}\left(j^{r+1}f(z)+t(j^{r+1}g(z)-j^{r+1}f(z)), \Sigma_{W, z}\right)\\
%&\geq\inf_{z\in D}\left(\mathrm{dist}(j^{r+1}f(z), \Sigma_{W, z})-\|j^{r+1}g(z)-j^{r+1}f(z)\|\right)\\
%&\geq \inf_{z\in D}\left(\mathrm{dist}(j^{r+1}f(z), \Sigma_{W, z})-\|f-g\|_{C^{r+1}(D, ,\R^k)}\right)\\
%&=\left(\inf_{z\in D}\mathrm{dist}(j^{r+1}f(z), \Sigma_{W,z})\right)-\|f-g\|_{C^{r+1}(D,\R^k)}\\
%&=\delta_W(f, D)-\|f-g\|_{C^{r+1}(D,\R^k)}>0.
%\end{align}
\end{proof}

\subsection{Semicontinuity of Betti numbers}
The key ingredient in the proof of the topological bound of Theorem \ref{thm:introsingsemiwitdash} is the following result. Applied to the case of $W\subset J^r(D,\R^k)\subset \R^p$, $\f=j^rf$ and $\psi=j^rg$, it allows to control, from below, the Betti numbers of the singularity $Z=(j^rg)^{-1}(W)$, provided that $g$ is close enough in the $\mathcal{C}^r$ topology to a known function $f$, such tht $j^rf\pitchfork W$. In this sense, it plays a role analogous to that played by Lemma \ref{lemma:qt} in the proof of Theorem \ref{thm:mainapp}, with the significant difference that it requires only to know the $C^r$ distance between $f$ and $g$.
\begin{theorem}[Theorem $2$ from \cite{mttrp}]\label{thm:semiconti}
Let $\f\colon D\to \R^p$ be a $C^{1}$ function and let $W\subset \R^p$ be a smooth closed submanifold.
Assume that $\f\pitchfork W$ and that the set $Z=\f^{-1}(W)$ is contained in the interior of $D$. Let $E\subset E_1\subset \text{int}(D)$ be tubular neighborhoods of $Z$ with the property that $\overline{E}\subset E_1$\footnote{Only the \emph{existence} of $E_1$ is needed.}.
\begin{enumerate}
    \item\label{semi:itm:openness} Let the space $C^0(D,\R^p)$ be topologized by the standard $C^0$ norm. Define the set  $\mathcal{U}_{E,\f}$ as the connected component containing $\f$ of the set
    \be\label{semi:eq:U}
    \mathcal{U}_E=\left\{\psi\in\mathcal{C}^0(D,\R^p)\colon \psi^{-1}(W)\subset E\right\}.
    \ee
    Then $\mathcal{U}_{E,\f}$ is open in $C^0(D,\R^p)$.
    \item Let $\check{b}(Z)$ denote the sum of the \v{C}ech Betti numbers of $Z$. For any $g\in\mathcal{U}_{E,\f}$, we have 
    \be \label{eq:cc}
    b(Z)\le \check{b}\left(\psi^{-1}(W)\right).
    \ee
    % there exist abelian groups $G_i$, for each $i\in \mathbb{N}$, such that
    % \be\label{semi:eq:sumgru}
    %  \check{H}^i(\Tilde{A})\cong\check{H}^i(A)\oplus G_i.
    % \ee
\end{enumerate}
\end{theorem}
\begin{remark}
The statement reported in \cite{mttrp} is weaker than the one above, in the sense that it doesn't specifies the neighborhood $\mathcal{U}_{E,\f}$ and it requires that $\psi\pitchfork W$. The proof of this stronger form can be found in the second author's PhD thesis \cite{MSphd}. In the proof of \cite[Theorem 2]{mttrp}, the condition on $\psi\pitchfork W$ is needed in order to guarantee that a neighborhood of the zero set retracts to it. In this context, let us notice that we will use Theorem \ref{thm:semiconti} in the proof of Theorem \ref{thm:semiwitdash} and of Theorem \ref{thm:introsingsemiwitdash}, where $\psi$ will be a polynomial and $W$ a semialgebraic subset. In this case, being the preimage of a semialgebraic set via a polynomial map semialgebraic, the condition of having a neighborhood which retracts on it is already verified; moreover, for the same reason, on the right hand side of \eqref{eq:cc} one can take ordinary Betti numbers (as \v{C}ech cohomology coincides with the singular one). Finally, the description of the neighborhood as in \eqref{semi:eq:U}, even if not stated in the original theorem \cite[Theorem 2]{mttrp}, is evident from its proof.
\end{remark}

\section{Polynomial approximation of singularities}

\subsection{Proof of Theorem \ref{thm:mainapp}}
In the sequel we will need the following quantitative versions of Weierstrass' Approximation Theorem from \cite{BBL}.

\begin{theorem}[Theorem 2 from \cite{BBL}]\label{thm:bblapprox}
For every $r\geq 0$ there exists a constant $a_r(D)>0$ such that for every $f\in C^{r}(D, \R)$ and for every $d\geq0$ there is a polynomial $p_d(f)\in \R[x_1, \ldots, x_n]$ of degree at most $d$ such that for every $\ell\leq \min\{r, d\}$:
\be\|f-p_d(f)\|_{C^{\ell}(D, \R)}\leq \frac{a_r(D)}{d^{r-\ell}}\cdot \|f\|_{C^r(D, \R)}.
\ee
\end{theorem}
\begin{corollary}
For every $r\geq 0$ there exists a constant $a_{r+2}(D)>0$ such that for every $f\in C^{r+2}(D, \R)$ and for every $\e>0$ there is a polynomial $p_d(f)\in \R[x_1, \ldots, x_n]$ of degree at most $d$ such that $\|f-p_d(f)\|_{C^{r+1}(D, \R)}<\e$ and:
\be
d\le \max\left\{r+1, \frac{\|f\|_{C^{r+2}(D, \R)}}{\e}\cdot a_{r+2}(D) \right\}.
\ee
\end{corollary}
With this result available, we prove now Theorem \ref{thm:mainapp}.

%\begin{theorem}\label{thm:mainapp}Let $W\subseteq J^{r}(D, \R^k)$ closed and semialgebraic. For every $f\in C^{r+2}(D, \R^k)$ with $j^{r}f\pitchfork W$ there exists a polynomial map $p=(p_1, \ldots, p_k)$ with each $p_i\in \R[x_1, \ldots, x_n]$ with 
%\be \deg(p_i)\leq \left(1+2a_{r+2}(D)\cdot \kappa^{(r+2)}_{W}(f, D)\right)\ee
%and such that:
%\be (D, j^{r}f^{-1}(W))\sim (D, j^{r}p^{-1}(W)).\ee
%(Here $a_{r+2}(D)>0$ is the constant given by Theorem \ref{thm:bblapprox}).
%\end{theorem}
Let $f=(f_1, \ldots, f_k)$ and use Theorem \ref{thm:bblapprox} to get for every $d\in \mathbb{N}$ a polynomial map $p_d(f)=(p_d(f_1), \ldots, p_d(f_k))$ such that:
\be\label{eq:ineq} \|f-p_d(f)\|_{C^{r+1}(D, \R^k)}=\left(\sum_{i=1}^k \|f_i-p_d(f_i)\|_{C^{r+1}(D, \R)}^2\right)^{1/2}\leq \frac{a_{r+2}(D)}{d}\cdot \|f\|_{C^{r+2}(d, \R^k)}.\ee
Choose now the approximating degree to be
\be d=\max\left\{r+1,\left\lfloor\frac{\|f\|_{C^{r+2}(D, \R^k)}}{\mathrm{dist}_{C^{r+1}}(f, \Delta_W)}\cdot 2a_{r+2}(D)\right\rfloor\right\}.\ee
Then
\be a_{r+2}(D)\cdot\frac{\|f\|_{C^{r+2}(D, \R^k)}}{\mathrm{dist}_{C^{r+1}}(f, \Delta_W)} <2 a_{r+2}(D)\cdot \frac{\|f\|_{C^{r+2}(D, \R^k)}}{\mathrm{dist}_{C^{r+1}}(f, \Delta_W)}\leq d,
%\leq  1+2 a_{r+2}(D)\cdot \frac{\|f\|_{C^{r+2}(D, \R^k)}}{\mathrm{dist}_{C^{r+1}}(f, \Delta_W)},
\ee
and the inequality \eqref{eq:ineq} becomes:
\be \|f-p_d(f)\|_{C^{r+1}(D, \R^k)} <\mathrm{dist}_{C^{r+1}}(f, \Delta_W).\ee
The conclusion follows now from Lemma \ref{lemma:qt}.
\subsection{Proof of Corollary \ref{coro:betti}}Let $p_d(f):D\to \R^k$ be the polynomial map given by Theorem \ref{thm:mainapp}, with:
\be d=\deg(p_d(f))\leq c_1(r,D)\cdot\max\left\{r+1,  \frac{\|f\|_{C^{r+2}(D, \R^k)}}{\mathrm{dist}_{C^{r+1}}(f, \Delta_W)}\right\}\ee
and such that
\be (D, j^rf^{-1}(W))\sim (D, j^{r}p_d(f)^{-1}(W)).\ee
Observe that $j^{r}p_d(f):D\to D\times J^{r}(n,k)$ is also a polynomial map, with each component of degree bounded by $d$. The set $W\subset J^{r}(D, \R^k)$ can be described by a quantifier-free Boolean formula without negations, involving a family of polynomials $\mathcal{Q}=\{q_1, \ldots, q_s\}$ with $s=s(W)$ and $\deg(q_i)\leq a_1(W)$, whose atoms are of the form $q_i\leq0$, $q_i\geq0$ or $q_i=0.$
In particular $j^rp_d(f)^{-1}(W)$ can also be described by a quantifier-free Boolean formula without negations, involving a family of $s$ polynomials of degrees  bounded by $d\cdot a_1(W)$, whose atoms are of the form $q_i\circ j^rp_d(f)\leq0$, $q_i\circ j^rp_d(f)\geq0$ or $q_i\circ j^rp_d(f)=0.$
We are therefore in the position of applying \cite[Theorem 1]{Basu} and we get the existence of a constant $a_2>0$ such that:
\begin{align} b(j^rf^{-1}(W))&=b(j^rp_d(f)^{-1}(W))\\
&\leq a_2^n s(W)^n(a_1(W)d)^n\\
&\leq  (a_2 s(W)a_1(W))^n\left(c_1(r,D)\cdot \max\left\{r+1, \frac{\|f\|_{C^{r+2}(D, \R^k)}}{\mathrm{dist}_{C^{r+1}}(f, \Delta_W)}\right\}\right)^n\\
&\leq c_2(W)\cdot \max\left\{r+1, \frac{\|f\|_{C^{r+2}(D, \R^k)}}{\mathrm{dist}_{C^{r+1}}(f, \Delta_W)}\right\}^n.\end{align}
\subsection{Proof of Theorem \ref{thm:introsingsemiwitdash}}\label{sec:singbetti}
Let us make the identification 
\be 
J^{r+1}(D,\R^k)\cong D\times \R^p\times \R^q,
\ee in such a way that the canonical projection $j^{r+1}_zf\to j^r_zf$ corresponds to the map $(z,\xi,\eta)\mapsto (z,\xi)$ and the so called ``source map'' $j^{r+1}_zf\mapsto z$ corresponds to the projection on the first factor.
\begin{definition}
Let $W\subset J^r(D,\R^k)\cong D\times \R^p$ be a smooth closed semialgebraic submanifold and let $\theta\in W$. 
\be 
C(T_\t W):= \left\{\a\in \R^{n\times (n+p)}\ \Big| \dim\left(\mathrm{im}(\a)+T_\t W\right)< n+p \right\}
\ee
\end{definition}
The set $C(T_\t W)$ is the semialgebraic subset of the space $\R^{n\times (n+p)}$ consisting of all the linear maps $\R^n\to\R^{n+p}$ that are not transverse to the vector space $T_\theta W\subset \R^{n+p}$. 

Let $W$ be compact. Suppose that $W\pitchfork J^r_z(D,\R^k)$ for all $z\in D$, where in our identification $J^r_z(D,\R^k)=\{z\}\times \R^p$, and define $W_z=W\cap J^r_z(D,\R^k)$. Then the set 
\be\label{eq:balldabliu} 
B_\e(W):=\left\{(z,\xi)\in J^r(D,\R^k)\ \Big|\  \mathrm{dist}\left((z,\xi),W_z\right)< \e \right\}
\ee
is a smooth open tubular neighborhood of $W$, for small enough $\e$ whose fibers are contained in the spaces $J^r_z(D,\R^k)$ (here we are using the transversality assumption). Let us denote by $\e_0(W)$ the supremum of the set of all $\e$ with such property.
%In this section we will use a modified version of $\delta_{W}(f,D)$.
\begin{definition}
%Let $W_z=W\cap \{z\}\times \R^p$.
Let $W$ be compact and such that $W\pitchfork J^r_z(D,\R^k)$ for all $z\in D$, so that $W_z=W\cap J^r(D,\R^k)$ is a closed submanifold.
For any $f\in C^{r+1}(D,\R^k)$ let $\pi_W(f,z)=\{\t\in W_z\colon \mathrm{dist}(j^rf(z),W)=|j^rf(z)-\t|\}$ and define
\be
\hat{\delta}_{W}(f,z):=\left(\mathrm{dist}(j^rf(z),W_z)^2+\min_{\t\in \pi_W(f,z)}\mathrm{dist}\left(d_z(j^rf),C(T_\t W)\right)^2 \right)^\frac12;
\ee
\be
\hat{\delta}_{W}(f,D):=\min\left\{\inf_{z\in D}\hat{\delta}_{W}(f,z);\inf_{z\in\de D}\mathrm{dist}(j^rf,W_z);\e_0(W)\right\};
\ee
%\be
%\hat{\kappa}^{(\ell)}(f,D)=\max\left\{\frac{\nurm \ell fk}{\hat{\delta}_{W}(f,D)},1\right\}.
%\ee
\end{definition}
\begin{remark}
Notice that $\hat{\delta}_W(f,D)>0$ if and only if $j^rf$ is transverse to $W$ at any point $z\in \text{int}(D)$ and moreover $j^rf^{-1}(W)\cap \de D=\emptyset$.
\end{remark}
\begin{lemma}\label{lem:tangentubordo}
Let $W\subset D\times \R^p$ be a compact smooth submanifold and assume that $W\pitchfork \{z\}\times\R^p$ for all $z\in D$. Let $(z,\xi)\in\de B_\e(W)$, with $0<\e<\e_0(W)$ and let $(z,w)\in W$ be such that $|\xi-w|=\e$. Then
\be
T_{(z,w)} W\subset T_{(z,\xi)} \de B_\e(W).
\ee
\end{lemma}
\begin{proof}
Let $\xi(t),w(s)\in\R^p$ be two smooth curves such that $\xi(0)=\xi$, $w(0)=w$, $(z,w(s))\in W$ and $(z,\xi(t))\in \de B_\e(W)$. Then for all $s,t$ we have
\be 
|\xi(t)-w(s)|\ge \mathrm{dist}\left((z,\xi(t)),W_z\right)=\e=|\xi(0)-w(0)|,
\ee
so that differentiating the function $(s,t)\mapsto |\xi(t)-w(s)|^2$ at the point $(0,0)$, we get that $\langle(\xi-w),\dot{w}\rangle=0$ and $\langle(\xi-w),\dot{\xi}\rangle=0$. The arbitrariness in the choice of the two curves implies that
\be 
T_{(z,\xi)}W\cap (\{0\}\times\R^p)\subset \{0\}\times (\xi-w)^\perp =T_{(z,\xi)}\de B_\e(W)\cap (\{0\}\times\R^p);
\ee
here on the right we have an equality instead than an inclusion for dimensional reasons.

Now let $(z(t),w(t))\in W$ be any smooth curve in $W$ such that $z(0),w(0)=z,w$. %Since $\de B_\e(W)\cap (\{z(t)\}\times \R^p)$ is a tubular neighborhood of $W_{z(t)}$
Given the tubular neighborhood structure of $\de B_\e(W)$ (see the discussion after equation \eqref{eq:balldabliu}), there also exists a smooth curve $(z(t),\xi(t))\in \de B_\e(W)$ with $\xi(0)=\xi$ and of course $(\dot{z},\dot{\xi})\in T_{(z,\xi)}\de B_\e(W)$. Thus
\be 
|\xi(t)-w(t)|\ge \mathrm{dist}\left(\xi(t),W_{z(t)}\right)=\e=|\xi(0)-w(0)|,
\ee
so that differentiating at $t=0$ we get that $\dot{\xi}-\dot{w}\in (w-\xi)^\perp$ meaning that the vector $(0,\dot{\xi}-\dot{w})$ belongs to the tangent space $T_{(z,\xi)}\de B_\e(W)$. We conclude that $(\dot{z},\dot{w})\in T_{(z,\xi)}\de B_\e(W)$. Since the latter construction can be repeated for any tangent vector $(\dot{z},\dot{w})\in T_{(z,w)}W$, this concludes the proof.
\end{proof}
% \begin{theorem}\label{thm:singsemiwitdash}
% Let $W\subseteq J^{r}(D, \R^k)$ be a smooth, compact and semialgebraic submanifold with the property that $W\pitchfork J^r_z(D,\R^k)$ for all $z\in D$. 
% There exists a constant $c_3(W)>0$ such that for every $f\in C^{r+1}(D, \R^k)$ with $j^{r}f\pitchfork W$ and $j^rf^{-1}(W)\cap \de D =\emptyset$ we have:
% \be b(j^{r}f^{-1}(W))\leq c_3(W)\cdot\left( \hat{\kappa}_W^{(r+1)}(f, D)\right)^n.\ee
% \end{theorem}
\begin{proof}[Proof of Theorem \ref{thm:introsingsemiwitdash}]
Consider the set $B_\e(W)$ defined as in \eqref{eq:balldabliu} and let us prove that
\be\label{eq:claimBWtransv}
j^rf\pitchfork \de B_\e(W), \qquad {\forall \e\in(0,\hat{\delta}_{W}(f,D))}.
\ee
By contradiction assume that there is a point $z\in D$ such that $j^rf(z)\in\de B_\e(W)$, but $j^rf$ is not transverse to $\de B_\e(W)$ at $z$ i.e.
\be\label{eq:nontransvjetto}
\mathrm{im}\left(d_z(j^rf)\right)\subset T_{j^rf(z)} \de B_\e(W).
\ee
Since $\e<\e_0(W)$, there exists a (unique) point $\t\in W_z$ such that $\|j^rf-\t\|=\e$ i.e. $\{\t\}=\pi_W(j^rf(z))$. By Lemma \ref{lem:tangentubordo} we have that $T_{j^rf(z)}\de B_\e(W)\supset T_\t W$  and this, together with \eqref{eq:nontransvjetto}, implies that the linear map $d_z(j^rf)\colon \R^n\to \R^{n+p}$ belongs to the critical set $C(T_\t W)$. This leads to a contradiction:
\be 
\hat{\delta}_{W}(f,z)\le \mathrm{dist}\left(j^rf(z),W_z\right)=\e< \hat{\delta}_{W}(f,z),
\ee
proving \eqref{eq:claimBWtransv}.

Fix $0<\e<\hat{\delta}_W(f,D)$. From the transversality condition \eqref{eq:claimBWtransv} it follows that the set $E=j^rf^{-1}(B_\e(W))$ is a smooth tubular neighborhood of the submanifold $Z=j^rf^{-1}(W)$ in the smooth manifold $\text{int}(D)$. The fact that $E$ is contained in the interior of $D$ is due to the inequality
\be
\e<\hat{\delta}_{W}(f,D)\le \inf_{z\in \de D}\mathrm{dist}(j^rf,W_z).
\ee
Now let $g\in\R[z_1,\dots,z_n]^k$ be a polynomial such that $\nurm r{f-g}k \le \e$. Then $j^rg^{-1}(W)\subset E$ indeed if $j^rg(z)\in W$, then
\be
\mathrm{dist}(j^rf(z),W_z)\le\mathrm{dist}(j^rf(z),j^rg(z))+ \mathrm{dist}(j^rg(z),W_z)\le \e+0,
\ee
hence $j^rf(z)\in B_\e(W)$. The same is true for all maps $g_t=tg+(1-t)f$, therefore the map $g$ belongs to the set $\mathcal{U}_{E,{j^rf}}\subset C^0(D,\R^{n+p})$ defined in the statement of Theorem \ref{thm:semiconti} and thus
\be\label{eq:semising}
b(Z)\le b\left(j^rg^{-1}(W)\right).
\ee
(Since $j^rg$ is a polynomial, we can use the singular homology Betti numbers.)
By Theorem \ref{thm:bblapprox} we can assume that each component of $g$ has degree smaller than an integer number $d\le a_{r+1}(D){\nurm {r+1}fk}\e^{-1}$ or, if the latter is less than $r$, $d\le r$ so that since $W$ is semialgebraic and the map $j^rg$ is polynomial there exists a constant $a(W)>0$ depending only on $W$ such that
\be\label{eq:bettideg}
b\left(j^rg^{-1}(W)\right)\le a(W)d^n\le a(W)\max\left\{r,a_{r+1}(D){\nurm {r+1}fk}\e^{-1}\right\}^n.
\ee
Combining \eqref{eq:semising} and \eqref{eq:bettideg} and from the arbitrariness of $\e$, we obtain the thesis. 
% Notice also that in the case in which $\left(a_{r+1}(D){\nurm {r+1}fk}\e^{-1}\right)\le 1$ for all $\e<\hat{\delta}_W(f,D)$, it follows that $\kappa^{(r+1)}(f,D)=1$, hence we can conclude by taking $d=1$.
\end{proof}
\section{What is the degree of a smooth hypersurface?}
In this section we consider the case in which $j^rf^{-1}(W)$ is just the zero set of the  function $f\colon D\to \R$. In this case $W=W_{0}$ is the subset of the space  $J^0(D,R)\cong D\times \R$ corresponding to $W_{0}=D\times\{0\}$. 

In this particular case we have that $\delta_{W_0}(f,D)=\mathrm{dist}(f,\Delta_{W_0})$ (see Proposition \ref{prop:nextprop} below) and for the rest of the current section we will denote this quantity by $\delta(f,D)$. 

Notice that under the identification $J^1(D,\R)\cong D\times \R\times \R^n$, we have that $\Sigma_{W_0,z}=\{z\}\times \{0\}\times\{0\}$ for every $z\in \mathrm{int}(D)$, while $\Sigma_{{W_0},z}=\{z\}\times \{0\}\times (T_z\de D)^\perp$ if $z\in\de D$, therefore the distance of $f$ from the discriminant $W_0$ can be expressed by the following formula.
\be\label{eq:deltahyp}
\delta(f,D)=\min\left\{\inf_{z\in D}\left(|f(z)|^2+|\nabla f(z)|^2\right)^\frac12,\inf_{z\in \de D} \left(|f(z)|^2+|\nabla (f|_{\de D})(z)|^2\right)^\frac12\right\}.
 \ee
\begin{proposition}\label{prop:nextprop}
$\delta_{W_0}(f,D)=\mathrm{dist}_{C^1}(f,\Delta_{W_0}).$
\end{proposition}
\begin{proof}
 Let us make the identification $J^1(D,\R)\cong D\times \R\times \R^n$.% Then for every $z\in \mathrm{int}(D)$ we have that $
%\Sigma_{W_0,z}=\{z\}\times \{0\}\times\{0\}$, 
%while, if $z\in\de D$, then $\Sigma_{{W_0},z}=\{z\}\times \{0\}\times (T_z\de D)^\perp$.
 By Definition, there exists a point $x\in D$ and a jet $j_xg=(x,r,v)\in J^1_x(D,\R)$ such that $j^1(f+g)(x)\in \Sigma_{W_{0},x}$ and
 \be
 \delta_{W_0}(f,D)=\|(r,v)\|.
 \ee
Fix $\e>0$ and let $\rho_\e\colon \R\to [-\e,\e]$ be a smooth function such that $\max_{t\in\R}|\rho_\e'(t)|=\rho_\e'(0)=1$. For instance $\rho$ can be constructed as follows:
\be \rho_\e(t)=\begin{cases}
                                   t & \text{if $ -\frac{\e}{2}\le t\leq \frac{\e}{2}$} \\
            
            \textrm{is increasing and convex} & \text{if $t\le -\frac{\e}{2}$} \\                       \textrm{is increasing and concave} & \text{if $\frac{\e}{2}\leq t$} \\
 \e & \text{if $|t|\geq 2\e$}.
  \end{cases}\ee
 Define now a new function $h\in C^1(D,\R)$ such that $h(z)=r+\rho_\e(v^T(z-x))$.
 Then the function $f+h$ belongs to the discriminant set $D_{W_0}$ since $j^1(f+h)(x)=j^1(f+g)(x)\in \Sigma_{W_{0},x}$, therefore
 \be
 \begin{aligned}
 \mathrm{dist}_{C^1}(f,\Delta_{W_0})&\le \|h\|_{C^1(D,\R)}\\
 &\le \sup_{z\in D}\left(|r+\rho_\e(z)|^2+|\rho_\e'(v^T(z-x))v|^2\right)^{\frac12}\\
 &\le \left(|r+\e|^2+|v|^2\right)^\frac12.
 \end{aligned}
 \ee
 Thus, from the arbitraryness of $\e>0$, we conclude that
 \be
 \mathrm{dist}_{C^1}(f,\Delta_{W_0})\le (r^2+|v|^2)^\frac12=\delta_{W_0}(f,D).
 \ee
 Combining this with the general inequality \eqref{eq:constant} we obtain the thesis.
% \be 
%  \delta_{W_0}(f,D)=\min\left\{\inf_{z\in D}\left(|f(z)|^2+|\nabla f(z)|^2\right)^\frac12,\inf_{z\in \de D} \left(|f(z)|^2+|\nabla (f|_{\de D})(z)|^2\right)^\frac12\right\}.
% \ee
% Suppose that such minimum is realized by the former expression, then 
% \be
% \begin{aligned}
% \delta_{W_0}(f,D)&=\inf_{z\in D}\left(|f(z)|^2+|\nabla f(z)|^2\right)^\frac12\\
% &=\inf_{z\in D}\mathrm{dist}_{J^1(D,\R)}\left(j^1f(z),j^10(z)\right)\\
% &= \mathrm{dist}_{C^1}(f,0)\\
% &\ge \mathrm{dist}_{C^1}(f,\Delta_{W_0}).
% \end{aligned}
% \ee
% The last inequality is due to the fact that the function $0$ clearly belongs to the discriminant $\Delta_{W_0}$.

% In the other case, we have
% \be
% \begin{aligned}
% \delta_{W_0}(f,D)&=\inf_{z\in \de D} \left(|f(z)|^2+|\nabla (f|_{\de D})(z)|^2\right)^\frac12\\
% &=\inf_{z\in D}\mathrm{dist}_{J^1(D,\R)}\left(j^1f(z),j^10(z)\right)\\
% &= \mathrm{dist}_{C^1}(f,0)\\
% &\ge \mathrm{dist}_{C^1}(f,\Delta_{W_0}).
% \end{aligned}
% \ee
\end{proof}
% \be 
% \Delta_W=&\left\{f\in C^1(D,\R)\colon
% \begin{aligned} &\exists x\in D \text{ s.t. }f(x)=0, \nabla f(x)=0\\ &\text{ or}\\
% & \exists x\in\de D \text{ s.t. }f(x)=0,\nabla f(x)\perp T_x\de D
% \end{aligned}\right\}.
% \ee
\subsection{Proof of Theorem \ref{thm:condeq}}
Denote by $\rho=\rho(Z).$ We consider the function $\dZ:\R^n\to \R$ defined to be the \emph{signed} distance from $Z$. By \cite[Remark 2]{Foote}, if $Z$ is of class $C^1$ and with positive reach $\rho(Z)>0$, the function $\dZ$ is $C^1$ on the set $\{\dZ<\rho(Z)\}.$ 

We need to consider also an auxiliary function $g:\R\to \R$ of class $C^2$ such that $g(t)=-g(-t)$ for all $t\in \R$ and:
\be g(t)=\begin{cases}
                                   t & \text{if $0\leq t\leq \frac{1}{2}$} \\
                                   \textrm{is increasing and concave} & \text{if $\frac{1}{2}\leq t\leq \frac{7}{8}$} \\
 \frac{3}{4} & \text{if $t\geq \frac{7}{8}$}
  \end{cases}\ee
  The existence of such a function is elementary; it can be taken, for instance, to be piecewise polynomial. Denoting by $g_\rho$ the function $t\mapsto \rho\cdot  g(t/\rho)$, we set:
 \be \label{eq:dfnf}f(x)=g_\rho(\dZ(x)).\ee
 
 Let us start by estimating $\delta(f, D)$, for a disk $D_R$ with $Z\subset \mathrm{int}D_{R-\rho}$. Notice that this condition implies that, by construction, the function $f\equiv \frac{3}{4}\rho$ on $\R^n\backslash \mathrm{int}(D)$ and in particular:
 \be\label{eq:d} \delta(f, D)=\min\left\{\inf_{z\in D}\left(|f(z)|^2+\|\nabla f (z)\|^2\right)^{1/2}, \frac{3}{4}\rho\right\}.\ee
 Observe now that for every $x$ such that $t=\dZ(x)<\rho$ we have:
 \begin{align} |f(x)|^2+\|\nabla f(x)\|^2&=|g_\rho(\dZ(x))|^2+|g'_{\rho}(\dZ(x))|^2\cdot\|\nabla \dZ(x)\|^2\\
 &=|g_\rho(\dZ(x))|^2+|g'_{\rho}(\dZ(x))|^2\\
 \label{eq:red}&=\rho^2|g(t/\rho)|^2+|g'(t/\rho)|^2,
 \end{align}
 where in the second line we have used the fact that $\|\nabla\dZ\|\equiv 1.$
 In particular, partitioning the domain of definition of the function $g_\rho$, it follows that:
 \begin{align} \inf_{z\in D}|f(z)|^2+\|\nabla f (z)\|^2&=\inf\left\{1+\frac{\rho^2}{4}, \rho^2(3/4)^2, \inf_{\frac{1}{2}\rho\leq t\leq \frac{7}{8}t}\rho^2|g(t/\rho)|^2+|g'(t/\rho)|^2\right\}\\
 &\geq \inf\left\{1+\frac{\rho^2}{4}, \rho^2(3/4)^2, \inf_{\frac{1}{2}\rho\leq t\leq \frac{7}{8}t}\rho^2|g(t/\rho)|^2+ \inf_{\frac{1}{2}\rho\leq t\leq \frac{7}{8}t}|g'(t/\rho)|^2\right\}\\
 &=\inf\left\{1+\frac{\rho^2}{4}, \rho^2(3/4)^2, \frac{\rho^2}{4}\right\}\\
 &=\frac{\rho^2}{4}.
 \end{align}
 Together with \eqref{eq:d}, this gives:
 \be\label{eq:db} \delta(f, D)\geq \frac{\rho}{2}.\ee
 
 Let us now estimate $\nu_1(f , D)$. Again partitioning the domain and using \eqref{eq:red} and the fact that $|g'(t)|\leq 1$ for all $t$, we immediately get:
 \be\label{eq:nb} \|f\|_{C^1(D, \R)}\leq 1+\frac{3}{4}\rho.
 \ee
 Combining \eqref{eq:db} with \eqref{eq:nb} gives \eqref{eq:k1}.
 
 In order to get \eqref{eq:k2}, we will work with the further assumption that $Z$ is of class $C^2.$ Under this assumption the function $\dZ$ is $C^2$ on $\{\dZ<\rho(Z)\}$ and, using \eqref{eq:dfnf} we get:
 \be\label{eq:second}\frac{\partial^2f}{\partial x_i\partial x_j}(x)=\begin{cases}
                                   0 & \text{if $\dZ(x)>\frac{7}{8}\rho$} \\
                                   g_\rho''(\dZ(x))\partial_i \dZ(x)\partial_i \dZ(x)+g_{\rho}'(\dZ(x))\partial^2_{ij}\dZ(x)& \text{otherwise}   \end{cases}.
 \ee
Since the function $g$ is fixed with $\|g'\|\leq 1$ and $\|g''\|\leq a_1$ (for some constant $a_1>0$), we have
\be \|g_\rho'\|=\|g'\|\leq 1\quad \textrm{and}\quad \|g_{\rho}''\|=\frac{1}{\rho}\|g''\|\leq \frac{a_1}{\rho}.\ee
In particular, in order to estimate the absolute value of \eqref{eq:second}, we need an estimate for the Hessian of $\dZ$.
We use now the following fact \cite[Lemma 14.17]{GT}: given $x\in \{\dZ<\rho\}$ with $\dZ(x)=t$, the Hessian of $\dZ$ at $x$ has eigenvalues:
\be \beta_1(t)=\frac{-\lambda_1}{1-\lambda_1t},\ldots, \beta_{n-1}(t)=\frac{-\lambda_{n-1}}{1-\lambda_{n-1}t},\quad\beta_n(t)=0\ee
where $\lambda_1, \ldots, \lambda_{n-1}$ are the eigenvalues of the Weingarten map of the hypersurface $Z$ at $z(x)=\mathrm{argmin}_{z\in Z}\mathrm{dist}(z, x),$ i.e. the principal curvatures of $Z$ in $\R^n.$ 

The modulus of each of these eigenvalues can be estimated by 
\be |\lambda_i|\leq \frac{1}{\rho(Z)},\ee
and each ratio $|1-\lambda_it|^{-1}$ is smaller than $1$ if $\lambda_1\leq0$ and, if $\lambda_i> 0$, smaller than its value at $t=\frac{7}{8}\rho$ (the extremum of the interval where we have to estimate the function), which is:
\be \frac{1}{|1-\lambda_it|}\leq \frac{1}{\left|1-\lambda_i\frac{7}{8}\rho\right|}\leq 8,\ee
where we have used $0<\lambda_i\leq \frac{1}{\rho}.$
%Consequently:
%\be |\partial^2_{ij}\dZ(x)|\leq a_2\cdot  \sup_{i} |\lambda_i(t)|\leq\frac{a_3}{\rho}.
%\ee
Going back to \eqref{eq:second}, we have
\be\left|\frac{\partial^2f}{\partial x_i\partial x_j}(x)\right|\leq |g_\rho''(t)|+|g_\rho'(t)|\cdot |\partial^2_{ij}\dZ(x)|.
\ee
Observe now that for the construction of the function $g$ we need: 
\be g''(1/2)=g''(7/8)=0\quad \textrm{and}\quad \int_{1/2}^{7/8}g''(s)ds=-1.\ee
Since $7/8-1/2=3/8>1/3$, we can choose the function $g$ such that $|g''|\leq 3,$ which implies $|g_\rho''|\leq 3/\rho.$

From this we immediately deduce that
\begin{align} \|f\|_{C^2(D, \R)}&\leq\|f\|_{C^1(D, \R)}+\left(\sum_{i,j}|\partial_{ij}^2f|^2\right)^{1/2}\\
&\leq  1+\frac{3}{4}\rho+\left(\sum_{i,j}(|g_\rho''|+|\partial_{ij}^2\dZ|)^2\right)^{1/2}\\
&\leq  1+\frac{3}{4}\rho+\left(\sum_{i,j}(3/\rho+|\partial_{ij}^2\dZ|)^2\right)^{1/2}\\
&\leq 1+\frac{3}{4}\rho+\left(\sum_{i,j}2((3/\rho)^2+|\partial_{ij}^2\dZ|^2)\right)^{1/2}\\
&\leq 1+\frac{3}{4}\rho+\sqrt{2}\left(n^2 \frac{9}{\rho^2}+\sum_{i,j}|\partial_{ij}^2\dZ|^2\right)^{1/2}\\
&= 1+\frac{3}{4}\rho+\sqrt{2}\left(n^2 \frac{9}{\rho^2}+\|\partial_{ij}^2\dZ\|_F^2\right)^{1/2}\\
&= 1+\frac{3}{4}\rho+\sqrt{2}\left(n^2 \frac{9}{\rho^2}+\sum_{i=1}^n\lambda_i(\partial_{ij}^2\dZ)^2\right)^{1/2}\\
&\leq 1+\frac{3}{4}\rho+\sqrt{2}\left(n^2 \frac{9}{\rho^2}+\frac{n}{\rho^2}\right)^{1/2}\\
&\leq1+\frac{3}{4}\rho+\frac{5n}{\rho}.
\end{align}
 Together with \eqref{eq:db} this gives \eqref{eq:k2}.

\subsection{The isotopy degree of a smooth hypersurface}\label{sec:isodeg}
\begin{definition}Let $Z\subset D$ be a smooth, compact submanifold without boundary. We define the number $\di(Z, D)$, the \emph{isotopy degree} of $Z$ in $D$, as the minimum degree of a polynomial $p$ such $(D, Z)\sim (D, Z(p))$.
\end{definition}
\begin{remark}\label{remark:measure}We observe that the isotopy degree of a hypersurface defined by a polynomial of degree $d$ can be smaller than $d$. In fact, following \cite{DiattaLerario}, one can prove that \emph{for most} polynomials $p$ of degree $d$ we have
\be \di(Z(p))=O\left(\sqrt {d \log d}\right),\ee in the following sense.
First, we can put a gaussian measure $\mu$ on the space of polynomials of degree $d$ by defining for every open set $A$ in $\R[x_1, \ldots, x_n]$ its measure by
\be \mu(A)= \frac{\int_A e^{-\|p\|^2_{\mathrm{FS}}}\mathrm{d}\lambda}{\int_{\R[x_1, \ldots, x_n]} e^{-\|p\|^2_{\mathrm{FS}}}\mathrm{d}\lambda},
\ee
where $\lambda $ is the Lebesgue measure on $\R[x_1, \ldots, x_n]\simeq \R^N$ (the identification is made through the list of coefficients) and $\|\cdot\|_{\mathrm{FS}}$ denotes the Fubini-Study norm: this norm is induced by a scalar product for which \be\left\{\sqrt{\frac{d!}{\alpha_1!\cdots \alpha_n!(d-|\alpha|)!}}\cdot x_1^{\alpha_1}\cdots x_n^{\alpha_n}\right\}_{|\alpha|\leq d}\ee form an orthonormal basis. 
Let now $p\in \R[x_1, \ldots, x_n]$ be a polynomial of degree $d$ and  denote by $x=(x_0,x_1, \ldots, x_n)$ and by $h\in \R[x_0, \ldots, x_n]$ the homogenization of $p$, the new variable being $x_0$. Then we can decompose $h$ into its spherical harmonics part 
\be h=h_d+\|x\|^2\cdot h_{d-2}+\|x\|^4\cdot h_{d-4}+\cdots.\ee
Denote by $\tilde{h}$ the projection of $h$ on the space of harmonics of degree at most $\sqrt{ b d \log d}$, where $b>0$ is a positive constant (defined in \cite[Proposition 6]{DiattaLerario}):
\begin{align} \tilde{h}(x_0, \ldots, x_n)&=\sum_{\ell\leq \sqrt{b d \log d}}\|x\|^{d-\ell}h_\ell(x_0, \ldots, x_n)\\
&=\|x\|^{d-\lfloor\sqrt{bd\log d}\rfloor}\sum_{\ell\leq \sqrt{b d \log d}}\|x\|^{\lfloor\sqrt{bd\log d}\rfloor-\ell}h_\ell(x_0, \ldots, x_n)\\
&=\|x\|^{d-\lfloor\sqrt{bd\log d}\rfloor}\cdot q(x_0, \ldots, x_n).\end{align}
Here $q$ is a homogeneous polynomial of degree bounded by $O(\sqrt{d\log d}).$
Then, by \cite[Theorem 7]{DiattaLerario}, there is a set $S_d\subset \R[x_0, \ldots, x_n]$ of homogeneous polynomials of degree $d$ such that $\mu(S_d)\to 1$ as $d\to \infty$ (i.e. with almost full measure) with the property that for every $h\in S_d$
\be (S^n, Z(h))\sim (S^n, Z(\tilde{h}))=(S^n, Z(q)).\ee
When we set $x_0=1$ in $q$ we obtain a polynomial $\tilde{p}$ on $\R^n\simeq \{x_0=1\}$, whose zero set is isotopic to $Z(p)$ and with degree $O(\sqrt{d \log d})$. 
\end{remark}

%\subsection{Polynomial approximation of hypersurfaces in the disk \comm{ok}}
%In the case of a regular hypersurface $Z(f)\subset D$ the quantity $\delta_W(f,D)$ becomes:
%\be\label{eq:delta} \delta_W(f, D)=\min\left\{\inf_{z\in D}\left(|f(z)|^2+\|\nabla f (z)\|^2\right)^{1/2}, \inf_{z\in \partial D}\left(|f(z)|^2+\|\pi_z\nabla f (z)\|^2\right)^{1/2}\right\}.\ee
%Here $\pi_z\nabla f (z)$ denotes the orthogonal projection of the gradient of $f$ at $z\in \partial D\simeq S^{n-1}$ on the tangent space to $\partial D$ at $z$. 
%In this special case we omit the ``$W$'' subscript in the notations both of $\delta(f, D)=\delta_W(f, D)$ and of $\kappa^{(r)}(f, D)=\kappa_W^{(r)}(f, D).$
%\begin{theorem}For every $f\in C^{2}(D, \R)$ such that $Z(f)$ is regular and contained in $\mathrm{int}(D)$ we have:
%\be \di(Z, D)\leq 1+2a_2(D)\kappa^{(2)}(f, D).\ee
%
%\end{theorem}
%\begin{proof}This is just a special case of Theorem \ref{thm:mainapp}, with the choice $r=1$ and $W=D\times \{0\}\subset J^0(D, \R)=D\times \R.$
%\end{proof}
\subsection{Proof of Theorem \ref{thm:dash}}
We will first prove the following preliminary estimate.

For $d, n>0$ denote by $\mathcal{Z}_{d, n}$ the set of rigid isotopy classes of pairs $(\R^n, Z(p))$ with $Z(p)\subset \R^n$ regular zero set of a polynomial $p\in \R[x_1, \ldots, x_n]$ of degree at most $d$. We claim that:
\be\label{eq:claim} \#\mathcal{Z}_{d,n}\leq 2 T (2T-1)^{\ell-1},\ee
where $T=(n+1)(d-1)^n$ and $\ell={n+d+1\choose n+1}.$

The cardinality of $\mathcal{Z}_{d,n}$ is bounded by the number of connected components of the complement of a discriminant in the space of polynomials. More precisely, denoting by $\Delta_{d, n}\subset \R[x_1, \ldots, x_n]\simeq \R^\ell$ the discriminant for $Z(p)$ being nonsingular, the number of rigid isotopy classes of $(\R^n, Z(p))$ is bounded by $b_0(\R^\ell\backslash \Delta_{d,n}).$ Denoting by $\widehat{\Delta}_{d, n}\subset S^N$ the one point compactification of $\Delta_{d,n}$, we have (using Alexander duality):
\begin{align} \#\mathcal{Z}_{d,n}&\leq b_0(\R^\ell\backslash \Delta_{d,n})\\&=b_0(S^\ell\backslash \widehat\Delta_{d,n})\\
&=\tilde{b}_0(S^\ell\backslash \widehat\Delta_{d,n})+1\\
&\leq \tilde{b}(S^\ell\backslash \widehat\Delta_{d,n})+1\\
&=\tilde{b}(\hat\Delta_{d,n})+1\\
&=b(\widehat\Delta_{d,n}).
\end{align}
In particular, in order to estimate $\#\mathcal{Z}_{d,n}$ it is enough to estimate the total Betti number of $\widehat\Delta_{d,n}.$ To this end we will use a Mayer-Vietoris argument and write:
\be \widehat{\Delta}_{d,n}=\Delta_{d,n}\cup \left(U_{\infty}\cap \widehat{\Delta}_{d,n}\right),
\ee
where $U_{\infty}\subset S^\ell$ is an open ball centered at the point at infinity in $ S^{\ell}$. Observe that $U_\infty\cap \widehat{\Delta}_{d,n}$ is contractible and 
\be \Delta_{d,n}\cap\left(U_{\infty}\cap \widehat{\Delta}_{d,n}\right)\sim Z(\mathrm{disc}_{d,n}, \|\cdot\|^2=R), \ee
meaning that the two spaces are homotopy equivalent; here $\mathrm{disc}_{d,n}$ is a polynomial on $\R^\ell$ whose zero set is $\Delta_{d,n}$ and $ \|\cdot\|^2=R$ defines a sphere on the same space (the boundary of $U_\infty$, viewed as a subset of $\R^\ell$). In particular both $\Delta_{d,n}$ and $\Delta_{d,n}\cap\left(U_{\infty}\cap \widehat{\Delta}_{d,n}\right)$ are described in $\R^N$ by polynomial equations of degree bounded by \cite[Proposition 7.4]{3264}: \be\deg(\mathrm{disc}_{d,n})=(n+1)(d-1)^n=T.\ee
Using Mayer-Vietoris and \cite{milnor}, it follows that:
\begin{align} b(\widehat{\Delta}_{d,n})&\leq b(\Delta_{d,n})+b\left(U_{\infty}\cap \widehat{\Delta}_{d,n}\right)+b\left(\Delta_{d,n}\cap\left(U_{\infty}\cap \widehat{\Delta}_{d,n}\right)\right)\\
&\leq T(2T-1)^{\ell-1}+1+T(2T-1)^{\ell-1}\\
&\leq 2T(2T-1)^{\ell-1},
\end{align}
which proves the claim \eqref{eq:claim}.

 For $d, n>0$ denote now by $\mathcal{C}_{d, n}$ the set of rigid isotopy classes of pairs $(D, C)$ with $C\subset \mathrm{int}(D)\subset \R^n$ a smooth component of the zero set $Z(p)$ of a polynomial $p\in \R[x_1, \ldots, x_n]$ of degree at most $d$. We claim now that:
\be \label{eq:claim2}\#\mathcal{C}_{d,n}\leq d(2d-1)^{n-1} \#\mathcal{Z}_{d,n}.\ee
In order to see this, observe first that if $C$ is a smooth component of $Z(p)$ we can slightly perturb $p$ within the space of polynomials of the same degree, without changing the rigid isotopy class of $(D, C)$ and making the whole zero set $Z(p)$ smooth (i.e. we can assume $Z(p)$ is smooth already). Also, notice that the inclusion $(D, C)\hookrightarrow (\R^n, C)$ gives a correspondence of rigid isotopy classes and we can work in the whole $\R^n$ instead of the interior of the disk $D$ (being the two ambient spaces diffeomorphic).
To every rigid isotopy class of pairs $(\R^n, Z(p))$ corresponds at most $b_0(Z(p))$ rigid isotopy classes of pairs $(\R^n, C)$ with $C$ a connected component of $Z(p)$ and this, together with the fact that $b_0(Z(p))\leq d(2d-1)^{n-1}$, gives \eqref{eq:claim2}.

By part Theorem \ref{thm:mainapp} we know that, given a pair $(D, Z(f))$ with $f=0$ regular, $Z(f)\subset \mathrm{int}(D)$ and $\kappa_2(f, D)\leq \kappa$, there exists a polynomial $p$ with
\be d=\deg(p)\leq (1+2a_2(D)\cdot \kappa):=k\ee such that the pairs $(D, Z(f))$ and $(D, Z(p))$ are rigidly isotopic. In particular, $\#(\kappa, D)$ is bounded by $\#\mathcal{C}_{\lfloor 1+2a_2(D)\cdot \kappa \rfloor, n}$. Combining \eqref{eq:claim2} and \eqref{eq:claim} we get:
 \begin{align} \#(\kappa,D)&\leq d(2d-1)^{n-1}2T(2T-1)^{\ell-1}\\
 &\le 2^nd^n(n+1)(d-1)^n\left(2(n+1)d^n\right)^{\frac{(d+n+1)^{n+1}}{(n+1)!}}\\
 &\le c(n)^{\frac{(k+n+1)^{n+1}}{(n+1)!}} k^{2n+\frac{(k+n+1)^{n+1}}{(n+1)!}}=:(*),
% &\leq 2^{n-1}d^n 2^{\frac{d^{n+1}}{(n+1)!}}((n+1)d^n)^{\frac{d^{n+1}}{(n+1)!}}\\
% &\leq c_2(n) c_3(n)^{d^{n+1}}d^n(d^n)^{d^{n+1}}\\
% &\leq c_2(n) c_3(n)^{(1+2a_2(D)\cdot \kappa)^{n+1}}(1+2a_2(D)\cdot \kappa)^{n}(1+2a_2(D)\cdot \kappa)^{n(1+2a_2(D)\cdot \kappa)^{n+1}}.
\end{align}
so that as $k\to+\infty$ we get:
\be
\begin{aligned}
(*)&\le k^{3\frac{(2k)^{n+1}}{(n+1)!}}\\
&\le k^{c'(n)k^{n+1}}\\
&= (1+2a_2(D)\cdot \kappa)^{c'(n)(1+2a_2(D)\cdot \kappa)^{n+1}}\\
&\le (c''(n)\kappa)^{c''(n)\kappa^{n+1}}\\
&\le \kappa^{c'''(n)\kappa^{n+1}}.
\end{aligned}
\ee
By the continuity of the expression $(*)$ with respect to $\kappa$, we conclude that there is a constant $C_1(n)$ such that if $\kappa >C_1(n)$ then $\#(\kappa,D)\le \kappa^{C_2(n)\kappa^{n+1}}$, where $C_2(n)=c'''(n)$ is the constant found before. This concludes the proof.
\subsection{Proof of Theorem \ref{thm:semiwitdash}}
For what regards the first part of the theorem, we will show that
\be 
b(Z(f))\le \left(a_1(D)\kappa^{(1)}(f,D)+1\right)^n,
\ee
where $a_1(D)$ is the constant given by Theorem \ref{thm:bblapprox}. This implies \eqref{eq:bound3}, since $\kappa^{(1)}(f,D)\ge 1$, by definition.

Fix $\e>0$. First, observe that if $\e< \delta(f,D)$, then $f$ and $f|_{\de D}$ have no critical value in the interval $(-\e,\e)$, from which it follows that the set $E=f^{-1}(-\e,\e)$ is entirely contained in the interior of $D$. Moreover $E$ is a tubular neighborhood of $Z(f)$, since by Morse theory $f^{-1}(-\e,\e)$ is diffeomorphic to $Z(f)\times (-\e,\e)$. To see that $E$ satisfies the hypotheses of theorem \ref{thm:semiconti}, define $E_1=f^{-1}(-\e_1,\e_1)$, where $\e<\e_1<\delta(f,D)$ and notice that then $E_1\subset \text{int}(D)$ is a tubular neighborhood of $Z(f)$ such that $\overline{E}\subset E_1$.

Let $p\in\R[x_1,\dots,x_n]$ be %the polynomial $p=w_{0,d}(f)$ from Theorem \ref{semi:lem:boh}, so that $\nu_0(f-p,D)< \e$ and 
a polynomial such that $\nurm 0{f-p}{}< \e$. By Theorem \ref{thm:bblapprox} we can assume that its degree $d$ satisfies the bound
\be\label{semi:eq:degreeeeee}
d-1\le a_1(D)\nurm{1}{f}{}\frac{1}{\e}
\ee
(take $p=p_d(f)$, where $d$ is the biggest positive integer such that \eqref{semi:eq:degreeeeee} is true).
Let $F_t=f+t(p-f)$ and call $F_t$ its restriction to $M=\text{int}(D)$. Consider the set $\mathcal{U}_E$ defined as in \eqref{semi:eq:U}, where $N=\R$ and $Y=\{0\}$. Suppose that $F_t\in\mathcal{U}_E$ for every $t\in [0,1]$, then we could apply Theorem \ref{thm:semiconti} to deduce that
\be 
b(Z(f))\le b(Z(p)) \le \left(\frac{a_1(D)\nurm1f{}}{\e}+1\right)^n
\ee
where the second inequality is due to the Milnor-Thom bound \cite{milnor} and to \eqref{semi:eq:degreeeeee}. The thesis now would follow by the arbitrariness of $\e$.
%, that
% \be 
% b(Z(f))\le \left(\frac{c_1}{\beta_1(f,D)}\right)^n.
% \ee

 Thus to conclude the proof it is sufficient to show that 
 %$F_t(D\smallsetminus E)\subset \R\smallsetminus \{0\}$, wich is equivalent to say that
 $F_t^{-1}(0)\subset E$. To see this, let $x\in D$ such that $F_t(x)=0$ and observe that then
\be
|f(x)|=|F_t(x)-t(p(x)-f(x))|\le \nurm 0{f-p}{}< \e.
\ee

Let us turn to the second statement of the theorem.
Assume for simplicity that $D$ is the standard unit disk in $\R^n$. We will show that for any given compact hypersurface $Z\subset D$ defined by a regular $C^1$ equation $f=0$ such that $f=1$ near $\de D$, there exists a sequence of smooth functions $f_m$ such that
\be 
b(Z(f_m))\ge \frac{b(Z)}{\kappa^{(1)}(f,D)^n}h(D)\kappa^{(1)}(f_m,D)^n,
\ee
where $h(D)$ is the infimum among the numbers $N\e^n$, such that there exists a collection of $N$ disjoint $n$-dimensional 
disks of radius $\e$ contained in $D$.

Extend $f$ to the whole space $\R^n$, by setting $f(x)=1$ for all $x\notin D$. Let $m\in\mathbb{N}$ and define $f_{m,z}\in \coo 1{\R^n}\R$ to be the function
\be 
f_{m,z}(x)=f\left(m(x-z)\right),
\ee
so that the submanifold $Z(f_{m,z})$ is contained in the interior of the disk of radius $m^{-1}$ centered in the point $z$ and it is diffeomorphic to $Z$.
Moreover we can observe that, with $m\ge 1$, we have the inequalities
\be\begin{aligned} 
\delta(f_{m,z},D)=\delta(f_{m,0,D}) &=\inf_{x\in D}\left(|f(x)|^2+m^{2}\|\nabla f (x)\|^2\right)^{1/2}\ge \delta(f,D);\\
\nurm1{f_{k,z}}{}=\nurm1{f_{k,0}}{} &=\sup_{x\in D}\left(|f(x)|^2+m^{2}\|\nabla f (x)\|^2\right)^{1/2}\le m\nurm0f{};
\end{aligned}\ee
therefore $\kappa(f_{m,z},D)\le m\kappa(f,D)$.

\begin{figure}
    \centering
    \includegraphics[scale=0.13]{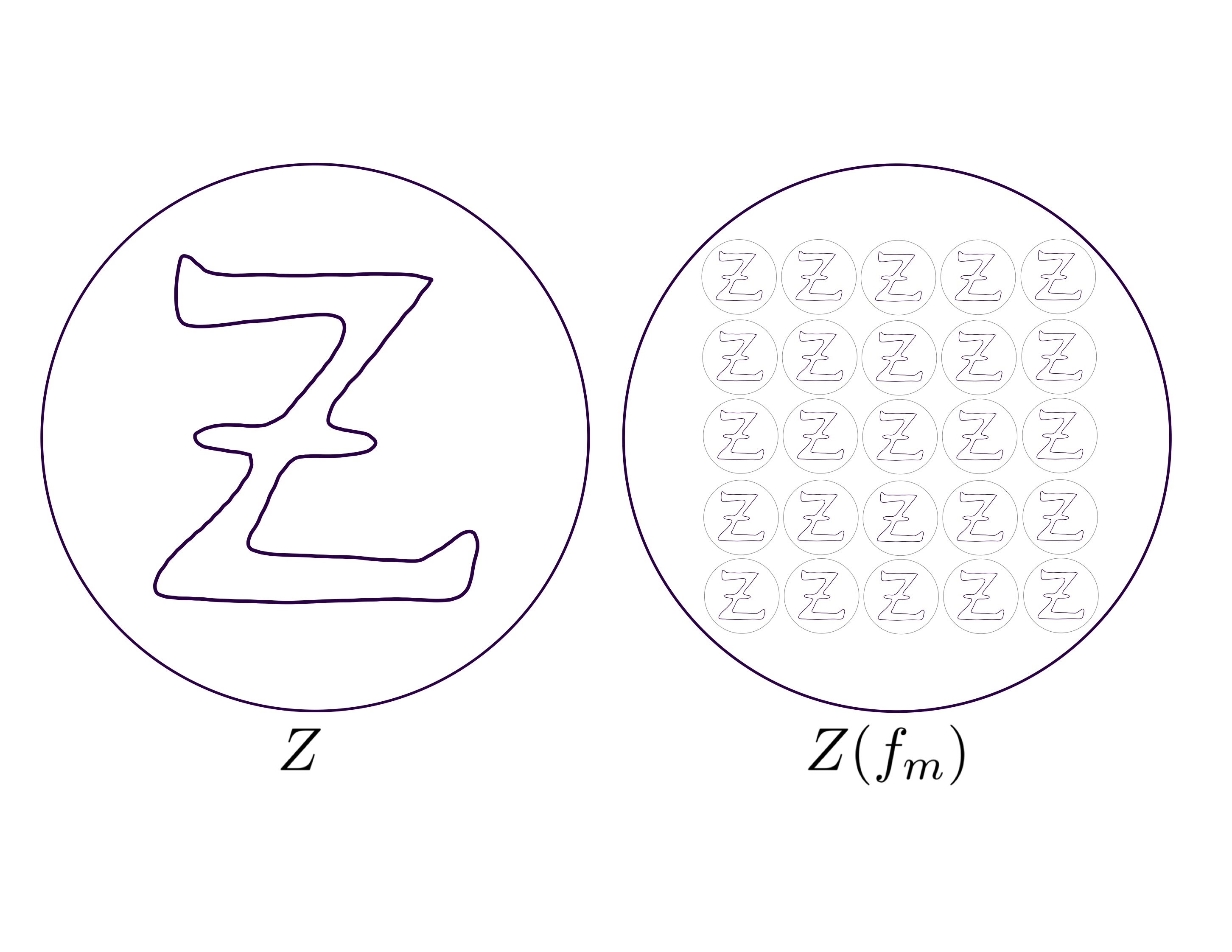}
    \caption{$Z(f_m)$, on the right, is the disjoint union of many copies of $Z$.}
    \label{fig:zzz}
\end{figure}

For any $m\in\mathbb{N}$, choose a finite family $I_m$ of 
points $z_{m,i}\in D$, such that the disks $D_{m,i}$, centered in $z_{m,i}$ and with radius $m^{-1}$, are disjoint. Since the (Hausdorff) dimension of $D^n$ is $n$, we can assume that the number of points in such a family is $\#(I_m)\ge h(D)m^n$.
Define the function $f_m\in \coo 1D\R$ as
\be 
f_m=1-\#(I_m)+\sum_{i\in I_m}{f_{m, z_i}},
\ee
so that $f_m$ coincides with $f_{m,i}$ on the disk $D_{m,i}$ and it is constantly equal to $1$ outside the union of all disks. 

It follows that $\kappa^{(1)}(f_m,D)\le m\kappa^{(1)}(f,D)$ and that the sequence $\kappa^{(1)}(f_m,D)$ is divergent as $m\to +\infty$.
Moreover, the zero set of $f_m$ is homeomorphic to a disjoint union of $\#(I_m)$ copies of $Z$, hence
\be 
b(Z(f_m))\ge \#(I_m)b(Z)\ge h(D)b(Z)m^n.
\ee
Putting these two observations together we conclude that 
\be 
b(Z(f_m))\ge h(D)b(Z)\left(\frac{\kappa^{(1)}(f_m,D)}{\kappa^{(1)}(f,D)}\right)^n.
\ee

\subsection{Proof of Corollary \ref{coro:NSW}}
 \begin{remark}\label{rem:NSW}Combining Theorem \ref{thm:semiwitdash} with Theorem \ref{thm:condeq}, one can obtain the estimate from Corollary \ref{coro:NSW} for the Betti numbers of $Z$. We observe that it is possible to obtain a similar estimate also from the work \cite{NSW}, let us sketch how.
 
 To start with, \cite[Proposition 3.1]{NSW} claims that there exists $\epsilon>0$ with 
 \be\label{eq:ee} \epsilon<\sqrt{\frac{3}{5}}\rho(Z)\ee
 such that, if $x_1, \ldots, x_m\in \R^n$ is a collection of points which are $\frac{\epsilon}{2}$--dense in $Z$, then the inclusion
 \be Z\hookrightarrow \bigcup_{i=1}^m B(x_1, \epsilon)=:U\ee
 is a homotopy equivalence. Since each ball is contractible, the set $U$ is also homotopy equivalent to the nerve complex $C$ of the cover $\{B(p_i, \epsilon), i=1, \ldots, m\},$ and in particular $b(Z)=b(C)$). 
 
 The number of vertices of this complex, i.e. $m$, can be chosen to be of the order:
 \be m\leq a_1(n)\frac{\mathrm{vol}(Z)}{\epsilon^{n-1}}.\ee
 Moreover we can chose the $\frac{\epsilon}{2}$--net to also satisfy the following: every ball $B(p_i, \epsilon)$ intersects at most $a_2(n)>0$ other balls from the family. In particular, each vertex of $C$ belongs to at most $a_2(n)$ cells, and:
 \be\label{eq:bNSW} b(Z)=b(C)\leq a_2(n) m\leq a_3(n)\frac{\mathrm{vol}(Z)}{\epsilon^{n-1}}.\ee
  Now, by Weyl's tube formula:
 \be \mathrm{vol}(\mathcal{U}(Z, \rho(Z)))=\mathrm{vol}(Z)\rho(Z)\leq \mathrm{vol}(D_{R+\rho}),\ee
 from which we get that 
 \be \label{eq:w}\mathrm{vol}(Z)\leq \frac{\mathrm{vol}(D_{R+\rho})}{\rho}.\ee
 Choose now $\epsilon=c_4 \rho(Z)$ such that \eqref{eq:ee} is satisfied. Then, combining \eqref{eq:bNSW} with \eqref{eq:ee} we get:
 \be b(Z)\leq a_3(n)\frac{\mathrm{vol}(Z)}{\epsilon^{n-1}}=a_5(n)\frac{\mathrm{vol}(Z)}{\rho(Z)^{n-1}}\leq a_6(n) \frac{\mathrm{vol}(D_{R+\rho})}{\rho(Z)^{n}}.\ee
 \end{remark}

\begin{proof}[Proof of Corollary \ref{coro:NSW}]
Observe that $\rho(Z)$ cannot be greater than the radius of $D$ unless $Z$ is empty, in which case there is nothing to prove. Define $D'$ to be the disk with a double radius than that of $D$, so that $Z$ and $D'$ satisfy the hypotheses of Theorem \ref{thm:condeq}, thus there exists a function $f$ such that
\be
\begin{aligned}
b(Z)&=b\left(Z(f)\right)\le \left(c_4(D')\cdot{\kappa_1(f,D')}\right)^n\le {c_4(D')}\cdot 2^n\left(1+\frac 1{\rho(Z)}\right)^n,
\end{aligned} 
\ee
where the first inequality is implied by Theorem \ref{thm:semiwitdash}. Taking $c_5(D)=2^nc_4(D')$ we conclude the proof.
\end{proof}
\appendix
\section{Global polynomial approximation of hypersurfaces} \label{sec:global}
The aim of this Section is to prove Theorem \ref{thm:quant}, which gives a quantitative bound for the degree of a polynomial approximating a smooth hypersurface on the whole $\R^n.$ We need firts the following Lemma.

\begin{lemma}\label{lemma:constants}Let $D\subset \R^n$ be a euclidean disk of radius $R$. For every $\ell, n,d\geq 0$ there exists a constant $c_\ell(n, d)>0$ such that if $p\in \R[x_1, \ldots, x_n]$ of degree $\deg(p)=d$, then:
\be |p(x)|\leq c_\ell(n, d)\, \nurm{\ell}{p}{}\, \left(\frac{\|x\|}{R}\right)^d\, (1+R^\ell)\quad \forall \|x\|\geq R>0.
\ee
\end{lemma}
\begin{proof}Let $V_{n,d}\subset \R[x_1, \ldots, x_n]$ be the space of polynomials of degree at most $d$. Since $V_{n,d}$ is a finite-dimensional vector space the norm $\|\cdot\|_{C^\ell(D_1,\R)}|_{V_{n,d}}$ (here $D_1$ is the unit disk) and the norm $\|\cdot\|_{\mathrm{coeff}}$, given by the ``maximum of the modulus of the coefficients'', are equivalent and there exists a constant $a_1(n,d)>0$ such that:
\be \|f\|_{\mathrm{coeff}}\leq a_{1}(n, d) \|f\|_{C^\ell(D_1,\R)}.\ee
Let now $p_R(y)=p(Ry)$. Then for every $\|y\|\geq 1$:
\begin{align}|p_R(y)|&\leq \dim(V_{n,d})\,\|p_R\|_{\mathrm{coeff}}\,\|y\|^d\\
&\leq \dim(V_{n,d})\,a_1(n,d)\,\|p_R\|_{C^\ell(D_1,\R)}\,\|y\|^d.
\end{align}
On the other hand we have:
\begin{align}\|p_R\|_{C^\ell(D_1,\R)}&=\max_{\|y\|\leq 1}\left(\sum_{|\alpha|\leq \ell}|\partial^{\alpha}p_R(y)|^2\right)^{1/2}\\
&=\max_{\|y\|\leq 1}\left(\sum_{|\alpha|\leq \ell}\left|\partial^{\alpha}p(Ry)R^{|\alpha|}\right|^2\right)^{1/2}\\
&\leq \max_{\|Ry\|\leq R}\left(\sum_{|\alpha|\leq \ell}\left|\partial^{\alpha}p(Ry)\right|^2\right)^{1/2}(R^\ell+1)
\\&=\|p\|_{C^\ell(D_1,\R)}(R^\ell+1).
\end{align}
This gives:
\be |p(x)|=|p_R(x/R)|\leq \dim(V_{n, d})\, a_1(n, d) \, (R^\ell+1) \|p\|_{C^\ell(D_R,\R)} \, \left(\frac{\|x\|}{R}\right)^d.
\ee
Defining the constant $c_\ell(n,d):=\dim(V_{n, d})\, a_1(n, d) $ gives the claim. \end{proof}

\begin{theorem}\label{thm:quant} Let $D$ be a disk of radius $R>0$ centered at the point $z_0$ and consider $f\in C^2(D, \R)$ such that the equation $f=0$ is regular in $D$ and $Z(f)\subset \mathrm{int}(D).$ Let also $\tau=\tau(f, D)>0$ be such that:
\be e^{-3\tau}R=\max_{z\in Z(f)}\|z-z_0\|\ee
and set $r=e^{-2\tau}R<R.$ Denote by $D_r$ the disk with the same center of $D$ and with radius $r$. Let $c_1(n, d)>0$ be the constant from Lemma \ref{lemma:constants} and define:
\be \tilde{\kappa}^{(\ell)}=\max\{\kappa^{(\ell)}(f, D), \kappa^{(\ell)}(f, D_r)\}.\ee
There exists a polynomial $p\in \R[x_1, \ldots, x_n]$ with 
\be \deg(p)\leq \max\left\{r+1,\tilde{k}^{(2)}\cdot 2a_2(D),
\frac{\log\tilde{k}^{(1)}+\log\frac1\tau+\log\left(
8+\frac{8}{Re^{-\tau}}+4Re^{-\tau}+4\tau\right)+\log\left(c_1(n,d)\right)}{\tau}
\right\}
\ee
such that:
\be (\R^n, Z(p))\sim (\R^n, Z(f)).\ee 
\end{theorem}
\begin{remark}
The estimate above is more interesting when $\tau\to 0$ and $\tilde{k}^{(1)}\to +\infty$, in which case it implies the simpler inequality
\be
\mathrm{deg}(p)\le C\frac{ \tilde{k}^{(2)}}{\tau^2}
\ee
\end{remark}
\begin{proof}[Proof of Theorem \ref{thm:quant}]
Assume that $D$ is centered at $0$. The first step of the proof is to argue as in the proof of Theorem \ref{thm:mainapp} and find a polynomial $p_0\in \R[x_1,\dots,x_n]$ such that 
\be 
\nurm{1}{p_0-f}{}<\frac12 \min\left\{\delta(f,D),\delta(f,D_r)\right\}=:\delta, 
\ee
so that $\left(D,Z(p_0)\cap D\right)\cong (D,Z(f))$ and, at the same time, $Z(p_0)\cap D\subset \text{int}(D_r)$. Thus we can assume that $p_0(D\smallsetminus D_r)>0$, since $p_0$ has no zeroes in that region.
Moreover, thanks to Theorem \ref{thm:bblapprox} we can estimate the degree $d$ of $p_0$:
\be \label{eq:12degree}
\begin{aligned}
d&\le \max\left\{r+1, \frac{\nurm2{f}{}}{\min\left\{\delta(f,D),\delta(f,D_r)\right\}}\cdot 2a_{r+2(D)}\right\}\\
&=
\max\left\{r+1, \max\left\{\frac{\nurm2{f}{}}{\delta(f,D)},\frac{\nurm2{f}{}}{\delta(f,D_r)}\right\}\cdot 2a_{r+2(D)}\right\}\\
&=\max\left\{r+1, \tilde{k}\cdot 2a_{r+2}(D)\right\}.
\end{aligned}
\ee

Now we have to modify $p_0$ in order to eliminate those components of its zero set $Z(p_0)$ that are not contained in $D$. To this end we define a new polynomial $p\in\R[x_1,\dots,x_n]$ such that
\be
p=p_0+a\left(\frac{|x|^2}{s^2}\right)^\ell,
\ee
for some $r<s<R$, $a>0$ and $\ell\in\mathbb{N}$. We need to take $\ell$ and $a$ so big that
\begin{enumerate}
    \item $p(x)>0$ for every $x\notin \text{int}(D)$. This ensures that $Z(p)=Z(p)\cap D$.
    \item $\nurm{1}{p-p_0}{}<\frac12 \delta(f,D)$, so that $\nurm{1}{p-f}{}< \delta(f,D)$ and thus $(D,Z(p)\cap D)\cong (D,Z(f))$ by Lemma \ref{lemma:qt}.
\end{enumerate}
Combined with the fact that $Z(f)\subset \text{int}(D)$, the two conditions above imply that $(\R^n,Z(p))\cong (\R^n,Z(f))$. Thus it remains to estimate the degree of such a polynomial $p$.

Let $x\notin D$. By Lemma \ref{lemma:constants} we have\be
\begin{aligned}
p(x)&\ge a\left(\frac{|x|}{s}\right)^{2\ell}- |p_0(x)|\\
&\ge
a\left(\frac{|x|}{s}\right)^{2\ell}-c_1(n,d)(1+s)\|p_0\|_{C^1(D_s,\R)}\left(\frac{|x|}{s}\right)^d,
\end{aligned}
\ee
where $D_s$ is the disk of radius $s$. Therefore condition $(1)$ is certainly satisfied for all $\ell> \frac12 d$ and
\be
\begin{aligned}
 \label{eq:cd1}a&\ge c_1(n,d)(1+s)\cdot2\|f\|_{C^1(D,\R)}\\
 &\ge c_1(n,d)(1+s)\cdot\left(\|f\|_{C^1(D_s,\R)}+\delta\right)\\
 &\ge c(n,d)(1+s^2)\cdot\|p_0\|_{C^1(D_s,\R)}.
\end{aligned}
\ee
\begin{lemma}
For any $\rho>1$ and $\ell\in \mathbb{N}$,
\be\label{eq:lemmalog}
\ell \log \rho\le \rho^\ell-1.
\ee
\end{lemma}
\begin{proof}
The function $\ell\mapsto \f(\ell)=\ell \log\rho-\rho^\ell+1$ takes the value $\f(0)=0$ at $\ell=0$ and has negative derivative for all $\ell>0$: 
\be
\f'(\ell)=\log\rho-\log\rho\cdot\rho^\ell=-(\rho^\ell-1)\log\rho\le 0.
\ee
\end{proof}
Applying the previous Lemma with $\rho=\frac{s}{r}>1$, we obtain the inequality
\be 
\ell\le \left(\left(\frac{s}{r}\right)^\ell-1\right)\frac{1}{\log\left(\frac{s}{r}\right)}.
\ee
Therefore
\be
\begin{aligned}
\|p-p_0\|_{C^2(D_r,\R)}&\le a\left(\frac{r}{s}\right)^{2\ell}+a2\ell\frac{r^{2\ell-1}}{s^{2\ell}}
\\
&=a\left(\frac{r}{s}\right)^{2\ell}\left(1+\frac{2\ell}{r}\right)
\\
&\le a\left(\frac{r}{s}\right)^{2\ell}\left(1+2\frac{\left(\frac{s}{r}\right)^{2\ell}-1}{r\log\left(\frac{s}{r}\right)}\right)
\\
&\le a\left(\left(\frac{r}{s}\right)^{2\ell}+\frac{2}{r\log\left(\frac{s}{r}\right)}\right)
\\
&\le a\left(1+\frac{2}{r\log\left(\frac{s}{r}\right)}\right)\left(\frac{r}{s}\right)^{2\ell}.
\end{aligned}
\ee
To ensure that condition $(2)$ is satisfied it is thus sufficient to assume that
\be\label{eq:cd2}
 a\left(1+\frac{2}{r\log\left(\frac{s}{r}\right)}\right)\left(\frac{r}{s}\right)^{2\ell}\le \delta,
\ee
which, combinined with \eqref{eq:cd1}, becomes
\be
\begin{aligned}
 \left(\frac{r}{s}\right)^{2\ell}&\le \left(1+\frac{2}{r\log\left(\frac{s}{r}\right)}\right)^{-1}\frac{\delta}{2c_1(n,d)(1+s)\|f\|_{C^1(D,\R)}}, \\ \text{i.e.} \quad
 \left(\frac{s}{r}\right)^{2\ell} &\ge \left(\frac{8}{r\log\left(\frac{s}{r}\right)}+4\right)c_1(n,d)(1+s)\tilde{k}^{(1)},\\
 \text{i.e.} \quad 2\ell &\ge \frac{\log\tilde{k}^{(1)}+\log\left(c_1(n,d)(1+s)\left(\frac{8}{r\log\left(\frac{s}{r}\right)}+4\right)\right) }{\log\left(\frac{s}{r}\right)}.
\end{aligned}
\ee
By taking $s=e^{-\tau}R\in (r,R)$, we get that $\frac{s}{r}=e^\tau$ and obtain the final formula formula
\be
\begin{aligned}
2\ell&\ge \frac{\log\tilde{k}^{(1)}+\log\left(c_1(n,d)\right)+\log\left((1+e^{-\tau}R)\left(\frac{8}{Re^{-\tau}\tau}+4\right)\right)}{\tau} 
\\
&=\frac{\log\tilde{k}^{(1)}+\log\frac1\tau+\log\left(
8+\frac{8}{Re^{-\tau}}+4Re^{-\tau}+4\tau\right)+\log\left(c_1(n,d)\right)}{\tau}.
\end{aligned}
\ee
We conclude that if we take $2 \ell\ge d$ and such that the previous inequality holds, then $p$ satisfies conditions $(1)$ and $(2)$. Combining this fact with the estimate \eqref{eq:12degree} on $d$ we conclude the proof.
\end{proof}

%\section{Examples \comm{to be updated}}\label{sec:examples}

\bibliographystyle{alpha}
\bibliography{What_is_the_degree_of_a_smooth_hypersurface}
\end{document}